\documentclass[11pt, a4paper]{article}
\usepackage{amssymb,amsmath,amsthm}
\usepackage{hyperref}
\usepackage{bm}
\usepackage[]{graphicx}
\usepackage{varwidth}
\usepackage{mathtools}

\usepackage{enumerate}
\usepackage{color}
\usepackage{verbatim}

\usepackage[utf8]{inputenc}

\usepackage{geometry}
\numberwithin{equation}{section}

\newcommand{\CL}{{\cal L}}
\newcommand{\CM}{{\cal M}}
\newcommand{\CN}{{\cal N}}
\newcommand{\CO}{{\cal O}}
\newcommand{\CQ}{{\cal Q}}
\newcommand{\CH}{{\cal H}}
\newcommand{\CA}{{\cal A}}
\newcommand{\CX}{{\cal X}}
\newcommand{\CT}{{\cal T}}
\newcommand{\CR}{{\cal R}}
\newcommand{\CU}{{\cal U}}

\newcommand{\R}{{\mathbb R}}
\newcommand{\N}{{\mathbb N}}
\newcommand{\E}{\mathbb{E}}

\newcommand{\prob}{\mathbb{P}}
\newcommand{\eps}{\varepsilon}

\newcommand\inner[2]{\langle #1, #2 \rangle}
\newcommand{\arcangle}{%
  \mathord{<\mspace{-9mu}\mathrel{)}\mspace{2mu}}%
}

\newtheorem{theorem}{Theorem}[section]
\newtheorem{lemma}[theorem]{Lemma}
\newtheorem{definition}[theorem]{Definition}
\newtheorem{remark}[theorem]{Remark}
\newtheorem{proposition}[theorem]{Proposition}
\newtheorem{assume}[theorem]{Assumption}
\newtheorem{conj}[theorem]{Conjecture}
\newtheorem{cor}[theorem]{Corollary}

\DeclareUnicodeCharacter{2212}{-}

\begin{document}

\title{Kink motion for the one-dimensional stochastic Allen--Cahn equation}
\author{Alexander Schindler \& Dirk Bl\"omker \\ Universit\"at Augsburg}
\date{\today}
\maketitle

\begin{abstract}
We study the kink motion for the one-dimensional
stochastic Allen-Cahn equation and its mass conserving counterpart.
Using a deterministic slow manifold,
in the sharp interface limit for sufficiently small noise strength we derive
an explicit stochastic differential equation for the motion of the interfaces,
which is valid as long as the solution stays close to the manifold.

On a relevant time-scale, where interfaces move at most by the minimal
allowed distance 
between interfaces, we show that the kinks behave approximately 
like the driving Wiener-process
projected onto the slow manifold, while in the mass-conserving case they are additionally coupled via the mass constraint.
\end{abstract}

\section{Introduction}

We study the stochastic Allen--Cahn equation (\ref{eq:AC}) together with its mass conserving modification (\ref{eq:mAC}) posed on
an one-dimensional domain driven by a small additive spatially smooth and white in time noise, which might depend on $\eps$, too.

As in the deterministic case (see \cite{CP89}) we use a deterministic slow manifold, 
which is parametrized by the position of the interfaces. 
The same manifold was also used for the Cahn--Hilliard equation in 
\cite{BaXun94, BaXun95} and its stochastic counterpart in \cite{ABK12, BYZ19}. 
Its key idea is to approximate an unknown true invariant manifold of the deterministic dynamic.
See also \cite{BJ14} for a more recent result on approximately invariant manifolds.
Here in Section \ref{AC:sec:slowmf} we introduce a different and simpler manifold, as due to the noise we observe a faster motion than the metastable motion of the deterministic case, where it is necessary to study in more detail also 
all exponentially small error terms.     

In Section \ref{sec:stab}, for the stochastic stability of the manifold
we show that with overwhelming probability a solution of the stochastic 
equation stays close to the manifold for extremely large times, unless 
the distance between two interfaces gets almost on the order of the atomistic interaction 
length $\varepsilon$. 
We are based on the method proposed in \cite{ABK12} but modify it in order to show stability in $L^2$ and $L^4$ spaces, 
which is necessary for the analysis of the dynamics along the manifold. Crucial  for stability is the 
spectral information of the linearized operator, which we recover from the single interface case on the whole real line in Section \ref{AC:sec:specAC}.

As long as the solutions of Allen--Cahn stay close to the manifold,
we show that the motion of the interfaces 
is given by a stochastic differential equation 
(see (\ref{eq:dh})). We do not show that the projection onto the manifold 
is globally well defined, but verify that we can always split 
the solution into a well defined diffusion process for the position on the manifold 
and the distance orthogonal to the manifold.  
However, we do not show that these coordinates are uniquely defined 
nor that the projection onto the manifold is well defined,
which might only be true closer to the manifold.  

We analyze the motion and show for the Allen--Cahn equation that on timescales where the interfaces
move on the order $\eps$, the     
motion is given by the driving Wiener process projected onto the slow manifold.
If the noise has no long-range correlation, then 
the motion of the interfaces are approximately independent.
For the mass conserving Allen--Cahn equation we verify a similar result,
but the interfaces are additionally coupled due to mass-conservation. 
The result for the Allen-Cahn equation was already studied in \cite{Shardlow00}
using formal analysis and numerical experiments.

We do not study what happens 
if two interfaces get on the order $\epsilon$ close to each other.
The obvious conjecture is that both annihilate, as shown in the deterministic case in \cite{CHEN04}, 
which is heavily based on the maximum principle.  
Conjectures and some details in the stochastic case can be found in \cite{Web14}.

In the higher dimensional case the situation is more complicated, 
as the motion of the interface should be driven by a free interface problem, 
which in general cannot be approximated by a finite dimensional manifold.
See \cite{ABK18, YZ19} for partial results or \cite{ABC94} in the deterministic case.

In special cases slow manifolds were 
used to study the motion of droplets (or bubbles) in various settings of the stochastic mass conserving Allen--Cahn equation \cite{ABBK15} (or \cite{ACF00} in the deterministic case) or the related Cahn--Hilliard equation
\cite{BS20} or \cite{AF98, AFK04} for the deterministic case.
On the other hand, for the Allen--Cahn equation the motion of droplets 
was not studied, as due to a lack of mass-conservation these droplets should collapse immediately.

Note that all the previous examples of droplet motion study the case of a single droplet, where the slow manifold is parametrized only by the position, 
while mass-conservation fixes the radius. The case of many bubbles 
does not seem to be studied rigorously yet. Although the slow manifold can be constructed and parametrized by position and radius of the droplets, 
it seems that the error terms by glueing single droplet solutions 
together are not exponentially small and thus pose an obstacle in obtaining rigorous results.        

Let us finally  remark, that the method of proof is related to the motion of traveling waves, where the manifold is given by translates of the wave profile.
The main difference here is that traveling waves do not move slowly on the manifold,
but travel at a constant speed.  

\section{Setting}

The stochastic Allen--Cahn equation on an one-dimensional domain driven by an additive spatially smooth and white in time noise $\partial_t W$ is given by
\begin{equation}
\label{eq:AC} \tag{AC}
\begin{cases}
\partial_t u \, = \, 
\eps^2 u_{xx} - f(u) + \partial_t W , \quad &0<x<1, \; t>0 \\
\;  u_x \, = \, 0, \quad  &x\in\{0,1\}.
\end{cases}
\end{equation}
Here, $0 < \eps \ll 1$ is a small parameter measuring the typical width of a phase transition, and~$f = F^\prime$ is the derivative
of a double well potential $F$. We assume that $F \in C^3(\R)$ is a smooth, even potential satisfying
\begin{center}
	\begin{varwidth}{10cm}
\begin{itemize}
\item[(S1)] $F(u) \geq 0$ and $F(u) = 0$ if, and only if, $u = \pm 1$,
\item[(S2)] $F^\prime$ has three zeros $\{ 0, \pm 1 \}$ and $F''(0) < 0, F''(\pm 1) > 0$,
\item[(S3)] $F$ is symmetric:  $F(u) = F(-u)\;\; \forall u \geq 0$.
\end{itemize}
    \end{varwidth}
\end{center}
The standard example is $F(u) = \tfrac14 (1-u^2)^2$ and thus $f(u) = u^3 - u$. For the simplicity of some arguments later, especially in determining the spectrum of the linearized operator, we focus
for the most part of this paper on this standard quartic potential, although the results remain valid for potentials satisfying the conditions
(S1)--(S3). For more details on this discussion, we refer to \cite{SchindlerPhD}.

For the moment, let us assume that $\int_0^1 \partial_tW(t,x) \, \mathrm{d}x = 0$ for all $t\geq 0$, i.e., in a Fourier series
expansion there is no noise on the constant mode. In contrast to the Cahn--Hilliard equation, \eqref{eq:AC} does not preserve
mass as
\[
\partial_t \int_0^1 u(t,x) \, \mathrm{d}x 
\, = \, \eps^2 \int_0^1 u_{xx} \, \mathrm{d}x - \int_0^1 f(u) \, \mathrm{d}x 
+ \int_0^1 \partial_tW(t,x) \, \mathrm{d}x 
\, = \, - \int_0^1 f(u) \, \mathrm{d}x.
\]
Throughout our analysis, we will therefore separately consider the mass conserving Allen--Cahn
equation~\eqref{eq:mAC}
\begin{equation}
\label{eq:mAC} \tag{mAC}
\begin{cases}
\displaystyle \partial_t u \, = \, 
\eps^2 u_{xx} - f(u) + \int_0^1 f(u) \, \mathrm{d}x 
+ \partial_t W , \quad &0<x<1, \; t>0 \\
\,\, u_x \, = \, 0, \quad  &x\in \{0,1\},
\end{cases}
\end{equation}
where we added the integral of $f$ over the interval $(0,1)$ to guarantee the conservation of mass. This can also be seen as an orthogonal projection of the right hand side onto the space orthogonal to the constants.

We denote the standard inner product in $L^2(0,1)$ by
$\inner{\cdot}{\cdot}$, i.e., $\inner{f}{g} = \int_0^1 f(x) g(x) \, \mathrm{d}x$, and the $L^2$-norm by~$\| \cdot \|$.
Other scalar products and norms appearing in subsequent sections will be endowed with a subindex. 
Moreover, we denote the Allen--Cahn operator by 
\[
\CL(\psi) \, = \, \eps^2 \psi_{xx} - f(\psi).
\]
We consider for a given ansatz function $u^h$ (defined later in Definition~\ref{def:profile}) 
the Taylor expansion of~$\CL$ around~$u^h$
\[
\CL(u^h + \psi) \, = \, \CL(u^h) + \CL^h \psi + \CN^h(\psi),
\]
where we define the linearization $\CL^h$ of $\CL$ at the ansatz function $u^h$ and the remaining nonlinear terms $\CN^h(v)$ by
\[
\CL^h \psi \, \coloneqq \, \mathrm{D}\CL(u^h) \psi \, = \, \eps^2 \psi_{xx} - f^\prime(u^h) \psi \quad \text{and} \quad
\CN^h (\psi) \, \coloneqq \, f(u^h) - f(u^h + \psi) + f^\prime(u^h) \psi.
\]
In the prototypical case of the quartic potential, this leads~to
$\CL^h \psi = \eps^2 \psi_{xx} + \psi - 3(u^h)^2 \psi$ and~${\CN^h (\psi)  =  - 3u^h \psi^2 - \psi^3}.$

In the case of the mass conserving Allen--Cahn equation, we have to assume that the Wiener process $W$ has mean zero. 
Furthermore, in order to apply It\^o-formula later, 
we need that solutions to both~\eqref{eq:AC} and~\eqref{eq:mAC} are sufficiently smooth in space and hence, we need that
the stochastic forcing $\partial_t W$ is sufficiently smooth in space, too.
The existence of solutions to both \eqref{eq:AC} and \ref{eq:mAC} at least in the case of a quartic potential is standard
and we will not comment on this in more detail. See \cite{DPZ14}.

In the following we will assume that $W$ which also depends on $\eps$ is given by a $\CQ$-Wiener process satisfying the following regularity properties.

\begin{assume}[Regularity of the Wiener process $W$]
Let $W$ be a $\CQ$-Wiener process in the underlying Hilbert space $L^2(\Omega)$,
$\CQ$ a symmetric operator, and $(e_k)_{k\in \N}$ an orthonormal basis with corresponding eigenvalues $\alpha_k^2$ such that
\[
\CQ e_k \, = \, \alpha_k^2 e_k \qquad \text{and} \qquad W(t) \, = \, \sum_{k \in \N} \alpha_k \beta_k(t) e_k,
\]
for a sequence of independent real-valued standard Brownian motions 
$\left\{ \beta_k \right\}_{k \in \N}$.
\\
We assume that the $\CQ$-Wiener process $W$ satisfies
\[
\mathrm{trace}_{L^2}(\CQ) = \sum_{k \in \N} \alpha_k^2 \eqqcolon \eta_{\eps} < \infty.
\]
Moreover, in the case of the mass conserving Allen--Cahn equation~\eqref{eq:mAC}, we suppose that $W$ takes its values in $L^2_0(\Omega)$, that is, 
\[
\int_0^1 W(t,x) \, \mathrm{d}x \, = \, 0 \quad \text{for all} \;\; t \geq 0.
\]
\end{assume}

Note that our results will thus depend on the squared noise strength $\eta_{\eps}$, which also depends on the parameter~$\eps > 0$.
The exact order in $\eps$ of $\eta_{\eps}$ will be fixed later in the main results.

\section{Construction of the slow manifold}
\label{AC:sec:slowmf}

In this section, we construct the fundamental building block for our
analysis, the slow manifolds~$\CM$ for \eqref{eq:AC} and $\CM_\mu$ for 
\eqref{eq:mAC}. Our construction of the slow manifolds is different 
to the deterministic case \cite{CP89}. We do not introduce a cut-off function to glue together the profiles connecting the stable phases
$\pm 1$. With this cut-off function, the authors took extra care of the exponentially small error away from the interface positions, which is crucial as the motion of the kinks in the deterministic case
is dominated by exponentially small terms. 

In our stochastic case, however, the (polynomial in $\eps$) noise strength dominates and hence, we are not concerned with these exponentially small terms. Thus we use our simplified mani\-folds, but we believe that as in the stochastic Cahn--Hilliard equation (see \cite{ABK12}) the original manifold of \cite{CP89} should work in our case, too.  
The main idea in our construction goes as follows:

We start with a stationary solution $U$ to~\eqref{eq:AC} on the whole
line $\R$, centered at $0$ and connecting the stable phases $-1$ and
$+1$ (Definition~\ref{def:multiAC:heteroclinic}). Using the exponential decay of~$U$ (Proposition~\ref{prop:decayhetero}), 
we introduce a rescaled version in the domain
$[0,1]$ in order to construct an ansatz function $u^h$,
which jumps from $\pm 1$ to $\mp 1$ in an $\CO(\eps)$-neighborhood of
the zeros $h_i$ (Definition~\ref{def:profile}).

Throughout our analysis, we fix the number $N+1$ of transitions. The presented results hold up to times, 
where the distance between two neighboring interfaces gets too small and  we thus cannot 
exclude the possibility of a collapse of two interfaces. This behavior of the stochastic equation
was not studied in full detail yet. In the deterministic case, we refer to the nice work by X.~Chen,~\cite{CHEN04}.
For some ideas and conjectures in the stochastic case, see the thesis of S.\ Weber \cite{Web14}.
Essentially, after an annihilation the number of transitions is reduced to $N-1$ and we can restart
our analysis on a lower-dimensional slow manifold.

Let us now define the building block of our manifold,
the heteroclinic connection on the whole real line connecting $-1$ and $+1$. 
\begin{definition}[The heteroclinic]
\label{def:multiAC:heteroclinic}

Let $U$ be the unique, increasing solution to 
\begin{equation}
\label{eq:heteroclinic}
U'' - f(U) \, = \, 0, \quad U(0) \, = \, 0, \;\; \lim_{x \to \pm \infty} U(x) \, = \, \pm 1.
\end{equation}
In the prototypical case $f(u) = u^3 - u$, we have the explicit solution
$
U(x)  =  \tanh ( x / \sqrt2 ).
$
\end{definition}
The function $U$ is the heteroclinic of the ODE connecting the stable points $-1$ and $+1$.
For a later discussion of the spectrum of the linearized Allen--Cahn operator, we need some relations between the heteroclinic $U$ and 
the potential $F$. We observe that if $U$ is a solution to~\eqref{eq:heteroclinic}, then
\[
\partial_x \left( U_x^2 - 2 F(U) \right) \, = \, 2 U_x \left( U_{xx} - F'(U) \right) \, = \, 0.
\]
From the boundary condition $U(0) = 0$, we conclude that solving equation~\eqref{eq:heteroclinic} is equivalent to solving the first-order ODE
\begin{equation}
\label{eq:equivheteroclinic}
U_x \, = \, \sqrt{2 F(U) }, \quad U(0) \, = \, 0, \;\;  \lim_{x \to \pm \infty} U(x) \, = \, \pm 1.
\end{equation}
By the assumptions on the potential $F$, we see that $\sqrt{F}$ is $C^1$ and hence, the solution to~\eqref{eq:equivheteroclinic} is unique.
Moreover, we observe that all derivatives of $U$ can be expressed as a function of~$U$. For instance, we have 
$U'' = F'(U)$, $U^{(3)} = F''(U) \sqrt{2 F(U)}$, and so on.
Also note that, due to the symmetry of $F$, the mirrored function $-U$ solves the same 
differential equation, but transits from~$ U( - \infty) = +1$ to $ U(+\infty) = -1$.
For some fine properties of $U$, we refer to the work of Carr and Pego
\cite{CP89}, which is based on~\cite{CGS84}.
Crucial for the construction of a slow manifold (cf.~Definition~\ref{def:multiAC:slowmf}) is that the heteroclinic $U$ together with its derivatives decay exponentially fast. The following 
proposition can be shown via phase plane analysis.
For a proof we refer to~\cite{AlFuSt96}. 

\begin{proposition}[Exponential decay of $U$]
\label{prop:decayhetero}
Let $U(x), \, x \in \R,$ be the heteroclinic defined by~\eqref{eq:heteroclinic}. There exist
constants~${c,C > 0}$ such that for $x \geq 0$
\[
 \vert 1 \mp U(\pm x) \vert \, \leq \, C e^{-cx}, \quad \vert U'(\pm x) \vert \, \leq \, Cc e^{-cx}, \quad  \text{and} \quad
 \vert U''(\pm x) \vert \, \leq \,  Cc^2 e^{-cx}.
\]
\end{proposition}

For $\xi \in \R$, we define a translated and rescaled version of $U$ by
\begin{equation}
\label{eq:rescaledprofile}
U(x; \, \xi, \, \pm 1) \, \coloneqq \, \pm \, U \left( \frac{x-\xi}{\eps} \right).
\end{equation}
One easily verifies that $U(\cdot \,; \, \xi, \pm 1)$ is a solution to the rescaled ODE ${\eps^2 \, U_{xx} - f(U) \, = \, 0}$,
centered~at $U(\xi; \xi, \pm 1) = 0$ and going from $\mp 1$ to $\pm 1$.
Due to the exponential decay of the heteroclinic, 
the rescaled profile $U(x \, ; \xi, \pm 1)$ is 
exponentially close to the states~$\pm 1$, if $x$ is at least 
$\CO(\eps^{1-})$-away from the zero $\xi$.

\begin{lemma}
\label{lem:expdecay}
Let $\kappa > 0$ and $0<\eps < \eps_0$. 
Then, uniformly for $ \vert x -\xi \vert > \eps^{1 - \kappa}$
\[
\vert U(x; \, \xi , \,\pm 1) \vert \, = \, 1 + \CO(\exp).
\]
Similar exponential estimates hold for the derivatives of 
$U(\cdot \, ; \,\xi,\pm 1)$.
\end{lemma}


Motivated by this lemma, 
we can construct for interface positions $h \in (0,1)^{N+1}$ 
with ${h_1 < h_2 < \ldots < h_{N+1}}$
profiles~${u^h : [0,1] \to \R}$ such that $u^h$ jumps from $\pm 1$ to
$\mp 1$ in a small neighborhood around $h_i$ of size~$\CO(\eps)$.
Locally around $h_i$, we prescribe
\[
u^h(x) \, \approx \, U(x \, ; \, h_i,(-1)^{i+1}).
\]
If $x$ is of order $\eps^{1-\kappa}$ away from $h_i$,
we assured in Lemma~\ref{lem:expdecay} that each profile~ $U(x; \, h_i, \, \pm 1)$ is close to
$\pm 1$ up to an exponentially small error. See Figure \ref{fig:potential}.
Thus we assume that the distance
between two neighboring interfaces and  to the boundary
is bounded from below by~$\eps^{1-\kappa}$ for some small $\kappa > 0$, 
and up to exponentially small error terms, we  define 
$u^h$ as the sum of profiles given by~\eqref{eq:rescaledprofile}. This
leads to the following definition.

\begin{definition}[The profile $u^h$]
\label{def:profile}
Fixing $\rho_\eps = \eps^\kappa$ for $\kappa > 0$ very small, we define the
set $\Omega_{\rho_\eps}$ of admissible interface positions in 
the interval $(0,1)$ by
\[
\Omega_{\rho_\eps} \, \coloneqq \, \left\{ h \in \R^{N+1} \,: \, 0 \, < \, h_1 \, < \, \ldots \, < \, h_{N+1} \, < \, 1, 
\,\, \max_{j=0.\ldots,N+1} \vert h_{j+1} - h_j \vert \, > \, \eps / {\rho_\eps}
\right\},
\]
where $h_0 \coloneqq - h_1$ and $h_{N+2} \coloneqq 2 - h_{N+1}$.
For $x \in (0,1)$ and $h \in \Omega_{\rho_\eps}$, we define 
\[
u^h(x) \, \coloneqq \, \sum_{j=1}^{N+1} \, U\left( x; \, h_j, \,(-1)^{j+1}\right)
+ \beta_N(x),
\]
where the normalization function $\beta_N(x)$ satisfies
$\beta_N(x)  =  \frac{(-1)^N - 1}{2} + \CO(\exp)$ and similarly for all derivatives. (cf.~Remark~\ref{rem:multiAC:renormalization}).
\end{definition}

Note that the positions $h_0$ and $h_{N+2}$ were
introduced to bound the distance of the interface positions from the
boundary $0$ and $1$. Moreover, it is straightforward to check that
the set $\Omega_{\rho_\eps}$ is convex, which is later used to bound 
the Lipschitz constant of the map $h\to u^h$.
\begin{figure}[ht]
   \centering 
  \includegraphics[scale=0.55]{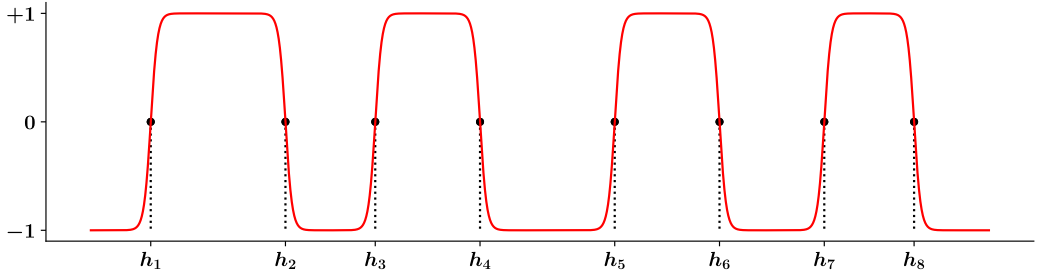}
  \caption{A sketch of the profile $u^h$ for $N=8$. The function is close to $\pm1$ up to sharp transitions around the interface positions.} 
  \label{fig:potential}
\end{figure}
\begin{remark}
\label{rem:multiAC:renormalization}
Let us comment on why we needed to add the normalization term~$\beta_N(x)$
in the definition of $u^h$.
Due to symmetry, we can assume that the multi-kink profile starts
in the phase~$u^h(0) = -1$. Depending on the parity of the number of transitions, we have 
to add a constant to assure this. As $h_j > \eps / {\rho_\eps}$, we obtain
by Lemma~\ref{lem:expdecay} in $x=0$
\[
\sum_{j=1}^{N+1} \, U\left( 0; \, h_j, \,(-1)^{j+1}\right)
\, = \, \sum_{j=1}^{N+1} (-1)^j  + \CO(\exp)
\, = \, -1 + \frac{1 - (-1)^N}{2}  + \CO(\exp).
\]
Therefore, we have to add the correction 
$\tfrac12 ((-1)^N -1) + \CO(\exp)$ to obtain $u^h(0) = -1$.

Moreover, we need to assure that $u^h$ satisfies Neumann boundary conditions. 
By Lemma~\ref{lem:expdecay}, the derivative of~$U(x; \, h_j, \pm 1)$ is exponentially small at $x \in \{0,1\}$. Hence, in order to correct
the boundary condition, we additionally have to add a function of order $\CO(\exp)$.
\end{remark}
Before we finally define the slow manifolds for the (mass conserving) Allen--Cahn equation (Definition~\ref{def:multiAC:slowmf}), we collect
some properties of the multi-kink configurations $u^h$.
\begin{proposition}[Properties of $u^h$]
\label{prop:uh}
The function $u^h$ is an almost stationary solution to~\eqref{eq:AC} in the sense that it satisfies the equation 
only up to an exponentially small error, that is, 
\begin{equation}
\label{eq:ACuh}
\eps^2 u^h_{xx} - f(u^h) \, = \, \CO(\exp), 
\qquad u^h_x(0) \, = \, 0,
\qquad u^h_x(1) \, = \, 0.
\end{equation}

For $i,j \in \{ 1,\ldots, N+1 \}$, we denote the partial derivatives of $u^h$ 
with respect to the $h$--variables by
$u^h_i = \partial_{h_i} u^h$, 
$u^h_{ij} = \partial_{h_i} \partial_{h_j}  u^h$,
and third derivatives accordingly. We have
\begin{equation}
\label{eq:uhx}
u^h_i(x) \, = \, U'(x; \, h_i, \, (-1)^{i+1}) + \CO(\exp)
\, = \, (-1)^i \,\frac1{\eps} \,U' \left( \dfrac{x-h_i}{\eps} \right)
+ \CO(\exp).
\end{equation}
Furthermore, the following estimates hold true in $L^2(0,1)$:
\begin{equation*}
\begin{split}
\inner{u^h_i}{u^h_j} \, &= \, \CX \eps^{-1} \delta_{ij} + \CO(\exp), \qquad \quad \;\;
\| u^h_{ij}\| \, = \, \CO (\eps^{-3/2}) \delta_{ij} + \CO(\exp), \\
\inner{u^h_{kk}}{u^h_k} \, &= \, \CO(\exp), \qquad  \quad \quad \text{and} \quad \quad
\| u^h_{kkk}\| \, = \, \CO(\eps^{-5/2}),
\end{split}
\end{equation*}
where $\CX \coloneqq \int_\R U'(y)^2 \, \mathrm{d}y$.
In $L^\infty(0,1)$, we have
\[
\|u^h\|_{\infty} \, = \, \CO(1) \qquad \text{and} \qquad
\|u^h_i\|_{\infty} \, = \, \CO(\eps^{-1}).
\]
Moreover all higher derivatives of $u^h$ are uniformly exponentially small
if the variables are mixed.
\end{proposition}

\begin{proof}

Equations \eqref{eq:ACuh} and \eqref{eq:uhx} follow directly from Definition~\ref{def:profile}, 
\eqref{eq:rescaledprofile} and Lemma~\ref{lem:expdecay}
By Lemma~\ref{lem:expdecay}, we also see that 
$U'(x \, ; \, h_i, (-1)^{i+1})$ is exponentially small for 
$\vert x - h_i \vert > \eps / {\rho_\eps}$ and thus we 
obtain $\inner{u^h_i}{u^h_j} = \CO(\exp)$ for $i \neq j$.
Moreover, the same argument implies that higher derivatives
with respect to different positions $h_i$ and $h_j$ are exponentially
small.
The uniform bounds are a direct consequence from the definition of $u^h$ and~\eqref{eq:uhx}. 
The $L^2$-norm of $u^h_k$ is given by 
\begin{equation*}
\begin{split}
\|u^h_k\|^2 \, &= \,
\eps^{-2} \int_0^1 U' \left( \dfrac{x-h_k}{\eps} \right)^2 \, 
\mathrm{d}x + \CO(\exp) \\
&= \, \eps^{-1} \int_{- h_k / \eps}^{(1-h_k) / \eps} U'(y)^2 \, 
\mathrm{d}y + \CO( \exp) 
\, = \, \eps^{-1} \int_\R U'(y)^2 \, \mathrm{d}y + \CO(\exp).
\end{split}
\end{equation*}
In the last step, we used that $h \in \Omega_{\rho_\eps}$ and 
$\vert U(x) \vert \leq c e^{-c\vert x \vert}$ 
by Proposition~\ref{prop:decayhetero}. Thus, we obtain
\[
\int_{- \infty}^{- h_k / \eps} U'(y)^2 \, \mathrm{d}y \, \leq \,
\int_{- \infty}^{- 1/{\rho_\eps}} U'(y)^2 \, \mathrm{d}y 
\, \leq \, c \int_{- \infty}^{- 1/{\rho_\eps}} e^{-c \vert y \vert} \, \mathrm{d}y 
\, = \, \CO(\exp),
\]
and with the same argument the integral at $\infty$ is exponentially small as well.
Analogously, the $n$-th derivative with respect to $h_k$ is then given by
\begin{equation*}
\| \partial^n_{h_k} u^h \|^2 \, = \, \eps^{-2n+1} \int_\R U^{(n)}(y)^2 \, \mathrm{d}y + \CO(\exp).
\end{equation*}
The mixed term can be estimated as follows:
\begin{equation*}
\begin{split}
\inner{u^h_{kk}}{u^h_k} 
\, &= \, \eps^{-3} \int_0^1 U'' \left( \dfrac{x-h_k}{\eps} \right)
U' \left( \dfrac{x-h_k}{\eps} \right) \, \mathrm{d}x 
+ \CO(\exp) \\
&= \, \frac12 \eps^{-2} \left[ U' \left( \dfrac{1-h_k}{\eps} \right)^2
- U' \left(-\dfrac{h_k}{\eps} \right)^2 \right] + \CO(\exp) 
\, = \, \CO(\exp). \qedhere
\end{split}
\end{equation*}
\end{proof}

We finally introduce the approximate slow manifolds for the 
stochastic (mass conserving) Allen--Cahn equation. The second manifold
will play an important role in the study of the mass conserving Allen--Cahn
equation~\eqref{eq:mAC}, while the first one will be used for the analysis of~\eqref{eq:AC} without this
constraint.
\begin{definition}[Slow manifolds]
\label{def:multiAC:slowmf}
For $\Omega_{\rho_\eps}$ and $u^h$ given by Definition~\ref{def:profile},
we define the \textit{approximate slow manifold} by
\[
\CM \, \coloneqq \, \left\{ u^h \, : \, h \in \Omega_{\rho_\eps} \right\}.
\]
Fixing a mass $\mu \in (-1,1)$, we define the 
\textit{mass conserving approximate manifold} by
\[
\CM_\mu \, \coloneqq \, \left\{ u^h \in \CM \, : 
\, \int_0^1 u^h(x) \, \mathrm{d}x \, = \, \mu \right\}.
\]
\end{definition}

Note that we have a global chart for $\CM$. Later in Lemma~\ref{lem:implfct} , we will see that this
also holds true for $\CM_\mu$, as it is the manifold $\CM$ intersected by
a vector space of codimension 1. 

We have to compute the tangent vectors for $\CM$ and $\CM_\mu$, since
we need them later in Definition~\ref{def:Fermi} to define a coordinate system around the slow
manifolds. We immediately see
that the tangent space of the slow manifold $\CM$  at $u^h$ with $h \in \Omega_{\rho_\eps}$ is given by
\[
\CT_{u^h} \CM \, = \, \mathrm{span} \left\{ u^h_i \, : \, i\,=\,1 \,,\,\ldots\,,\,N+1 \right\}.
\]
With Proposition~\ref{prop:uh} one checks readily that the tangent vectors $u^h_i$ have essentially 
(up to an exponentially small error) disjoint support and therefore, 
$\CT_{u^h} \CM$ is non-degenerate and has full dimension $N+1$.
For the second manifold $\CM_\mu$ we will see in the following lemma that, due to 
mass conservation, it is possible to reduce the parameter space 
$\Omega_{\rho_\eps}$ by one dimension. The proof is a simple argument based on the implicit function theorem and omitted. For details see \cite{BaXun94} or \cite{SchindlerPhD}.

\begin{lemma}
\label{lem:implfct}
There is a smooth map $h_{N+1} : [0,1]^N \to \R$ such that 
\[
u^h \in \CM_\mu \; \iff \; h = (\xi ,h_{N+1}(\xi)) \in \Omega_{\rho_\eps} 
\quad \text{with} \quad  \xi \, = \, (h_1, \, \ldots \,, h_N).
\]
Moreover, the partial derivatives of $h_{N+1}$ with respect to 
$h_i,\, i= 1,\ldots,N,$ are given by
\[
\dfrac{\partial h_{N+1}}{\partial h_i} \, = \, (-1)^{N-i} +  \CO(\exp).
\]
\end{lemma}

We can then write
\begin{equation}
\label{def:admissiblemasscon}
\CM_\mu \, = \, \left\{ u^h \, : \, h \in \CA_{\rho_\eps} \right\} \; \text{with}
\;
\CA_{\rho_\eps} \, \coloneqq \, \Big\{ (\xi,h_{N+1}(\xi))  \in \Omega_{\rho_\eps} \, : \xi\in[0,1]^N \Big\}.
\end{equation}
In the sequel, we denote the elements of $\CM_\mu$ by $u^\xi$. As before, we denote the partial derivatives of $u^\xi$ with 
respect to $\xi_i$ by $u^\xi_i$, and higher derivatives accordingly.\\
The tangent space of the mass conserving manifold $\CM_\mu$ at $u^\xi$ 
is given by
\begin{equation*}
\CT_{u^\xi} \CM_\mu 
\, = \,  \mathrm{span} 
\left\{ 
u^\xi_i \, = \, u^h_i + (-1)^{N-i} u^h_{N+1} + \CO(\exp) \, : \, i\, = \, 1,\ldots,N  \right\}.
\end{equation*}
Here, we used that by the chain rule and Lemma~\ref{lem:implfct}
\begin{equation*}
\dfrac{\partial u^\xi}{\partial \xi_i} \, = \,
\dfrac{\partial u^h}{\partial h_i} +
\dfrac{\partial h_{N+1}}{\partial h_i} \cdot
\dfrac{\partial u^h}{\partial h_{N+1}}  
\, = \, \dfrac{\partial u^h}{\partial h_i} + (-1)^{N-i} 
\dfrac{\partial u^h}{\partial h_{N+1}} + \CO(\exp).
\end{equation*}
This is a linear combination of tangent vectors of $\CM$.
Since the functions $u^h_i$ span an $(N+1)$-dimensional space and the
transformation matrix converting these functions into 
$\{u^\xi_1, \ldots, u^\xi_N \}$ has full rank $N$, we immediately
obtain that the tangent space $\CT_{u^h} \CM_\mu$ is 
non-degenerate.

\section{The linearized Allen--Cahn operator}
\label{AC:sec:specAC}

Important for the stability of the slow manifolds are spectral estimates concerning the linearization of the Allen--Cahn
operator at a multi-kink configuration. In more detail,
for $v$ orthogonal to the tangent space of $\CM$ or $\CM_{\mu}$,
we aim to bound the quadratic form $\langle \CL^h v,v \rangle$.
First, we consider the singular Sturm--Liouville problem
\begin{equation}
\label{eq:SturmLiouville}
Ly \, = \, y'' - f^\prime(U)y \, = \, \lambda y
\end{equation}
in $L^2(\R)$, where $U$ is the heteroclinic solution defined by~\eqref{eq:heteroclinic}. 
Note that the ODE~\eqref{eq:heteroclinic} directly implies that
$U'$ is an eigenfunction of $L$ corresponding to the eigenvalue
zero. As $U' > 0$, we also know that zero must be the largest eigenvalue.
The following description of the spectral behavior of $L$ orthogonal to $U^\prime$ is taken from~\cite{OR07}, Proposition~3.2.

\begin{lemma}[Spectral gap of the Allen--Cahn operator, \cite{OR07}, Proposition~3.2]
\label{lem:specgapliouville}
There exists a constant $\lambda_0 > 0$ such that if $v \in H^1(\R)$ satisfies
\[
\mathrm{(i)} \; v(0) \, = \, 0 \qquad \text{or} \qquad \mathrm{(ii)} \; \int_\R v(s) U^\prime(s) \, \mathrm{ds} \, = \, 0,
\]
then it holds true that
\[
\inner{Lv}{v}_{L^2(\R)} \, = \, \int_\R \left[ - v^\prime(s)^2  -  f^\prime(U(s)) v(s)^2 \right] \, \mathrm{ds} \, \leq \, - \lambda_0 \|v\|_{L^2(\R)}^2 .
\]
\end{lemma}
For the simplicity of some arguments, we focus in the remainder on the classical cubic potential $f(u) = u^3 - u$.  
In this case, one can show that $\lambda_0 = 3/2$ and $U\sqrt{U'}$ serves as corresponding eigenfunction of~\eqref{eq:SturmLiouville} (cf.~\cite{AlFuSt96}). 
The eigenfunction $U \sqrt{U'}$ has exactly one zero and hence it
corresponds to the second largest eigenvalue. As $ \lim_{\vert x \vert \to \infty} (1-3U^2) = - 2$, we also know by a
standard argument for Schrödinger operators that the essential spectrum
lies in the interval~$(- \infty, -2 ]$. For more details on the spectrum of Schrödinger operators, we refer to~\cite{HS96}. The standard arguments
for Sturm--Liouville problems can be found in~\cite{Walter98}.

With the spectral gap of Lemma~\ref{lem:specgapliouville} at hand, we consider the linearization of the Allen--Cahn operator at a multi-kink 
state~$u^h \in \CM$. The 
following theorem  gives a bound on the quadratic form orthogonal to the tangent space $\CT_h \CM$. Essentially, up to exponentially small 
terms, the support of the tangent  vectors $u^h_i$ is concentrated in a small neighborhood of width~$\eps$ around the zero $h_i$.
Hence, it is sufficient to study the quadratic form locally around each~$h_i$. After rescaling, we essentially arrive at the setting of 
Theorem~\ref{lem:specgapliouville} and the spectral gap of order~$1$ is transferred to our problem.

\begin{theorem}[Spectral gap for \eqref{eq:AC}]
\label{thm:gapAC}
Let $u^h \in \CM$ and $v \perp u^h_i$ for any~$ i=1, \ldots, N+1$. Then, for $\lambda_0$ given in Lemma~\ref{lem:specgapliouville}, we have
\[
\langle \CL^h v, v \rangle \, \leq \, \left(- \frac12 \lambda_0 + \CO({\rho_\eps}^2) \right) \Vert v \Vert^2.
\]
\end{theorem}

\begin{proof}
Since the minimal distance between the interfaces $h_i$ is bounded from 
below by~$\eps / {\rho_\eps}$ and the heteroclinic solution $U$ goes
exponentially fast to $\pm 1$ by Proposition~\ref{prop:decayhetero}, 
we find $0 < \delta_\eps < \frac12 \eps / {\rho_\eps}$ 
such that~$u^h = \pm 1 + \CO(\exp)$ on 
$\CR \coloneqq [0,1] \setminus \bigcup B_{\delta_\eps}(h_i).$
On the set $\CR$ we have
\begin{equation*}
\inner{\CL^h v}{v}_{L^2(\CR)} \, = \, - \eps^2 \int_{\CR} v_x^2 
- \int_{\CR} f'(\pm 1) v^2 + \CO(\exp) \|v \|_{L^2(\CR)}^2 
\leq \, \left( - f'(\pm 1) + \CO(\exp) \right) \| v \|_{L^2(\CR)}^2,
\end{equation*}
which is strictly negative as $f'(\pm 1) > 0$.

It remains to control the quadratic form on each
$B_{\delta_\eps}(h_i)$ and, without loss of generality, we may shift it 
to $h_i = 0$. 
Note that ${u^h(x) = U( \tfrac{x-h_i}{\eps}) + \CO(\exp)}$ on the set $B_{\delta_\eps}(h_i)$ by Proposition~\ref{prop:uh}.
Defining $\tilde{v}(x) \coloneqq v(\eps x)$, one easily computes for $h_i = 0$
\begin{align}
\label{proof:eqAC1}
\inner{\CL^h v}{v}_{L^2(B_{\delta_\eps})} 
\, = \, \eps \inner{L \tilde{v}}{\tilde{v}}_{L^2(B_{\delta_\eps/\eps})} + \CO(\exp) \| v \|^2_{L^2(B_{\delta_\eps})}.
\end{align}
Here, $L$ denotes the singular Sturm--Liouville operator defined 
by~\eqref{eq:SturmLiouville}.
After rescaling, we 
essentially have to bound the quadratic form $\inner{L\tilde{v}}{\tilde{v}}$ on the interval
$( - \delta_\eps / \eps, \delta_\eps / \eps) =: D_\eps$, which is a set of length of order
$\CO( {\rho_\eps}^{-1}) = \CO( \eps^{- \kappa})$.
To compare with the spectrum on the whole line, we define a cut-off 
function~${\phi \in C_c^\infty (D_\eps)}$ such that $0 \leq \phi \leq 1$
and $\phi \equiv 1$ on the set~$\left\{ f'(U) < C \right\}$ for some
$\lambda_0 < C < \sup_{-1 < x < 1} f^\prime(x)$. 
As~$| D_\eps | = \CO( {\rho_\eps}^{-1})$, we can also 
assume that uniformly~$\vert \phi_x \vert \leq C {\rho_\eps}$ 
and~$\vert \phi_{xx} \vert \leq C {\rho_\eps}^2$.
We obtain
\begin{align}
\label{proof:eqAC2}
\inner{L\tilde{v}}{\tilde{v}}_{L^2(D_\eps)} \, &= \, -\int_{D_\eps} \tilde{v}_x^2 \phi^2 - \int_{D_\eps} (1- \phi^2) \tilde{v}_x^2 
- \int_{D_\eps} \phi^2 f'(U) \tilde{v}^2 
- \int_{D_\eps} (1- \phi^2) f'(U) \tilde{v}^2 \nonumber   \\
&\leq \,  -\int_{D_\eps} \tilde{v}_x^2 \phi^2 
- \int_{D_\eps} \phi^2 f'(U) \tilde{v}^2 - C \int_{D_\eps} (1- \phi^2) \tilde{v}^2 \nonumber    \\
&= \, \int_{D_\eps} - \left(( \phi \tilde{v})_x \right)^2 + f'(U) ( \phi \tilde{v})^2 
+ 2 \int_{D_\eps} \phi_x \tilde{v}_x \phi \tilde{v} + \int_{D_\eps} \tilde{v}^2 \phi_x^2 
- C \int_{D_\eps} (1- \phi^2) \tilde{v}^2 \nonumber \\
&= \, \inner{L \phi \tilde{v}}{\phi \tilde{v}}_{L^2(\R)} - \int_{D_\eps} \tilde{v}^2 ( \phi \phi_x)_x + \int_{D_\eps} \tilde{v}^2 \phi_x^2 
- C \int_{D_\eps} (1- \phi^2) \tilde{v}^2  \nonumber  \\
&\leq \, - \frac12 \lambda_0 \int_{D_\eps} \phi^2 \tilde{v}^2 
- \int_{D_\eps} \tilde{v}^2 ( \phi \phi_x)_x + \int_{D_\eps} \tilde{v}^2 \phi_x^2 
- C \int_{D_\eps} (1- \phi^2) \tilde{v}^2 + \CO(\exp) \nonumber  \displaybreak \\
&\leq \,  \left(- \frac12 \lambda_0 + \CO({\rho_\eps}^2)\right) \int_{D_\eps} \tilde{v}^2 
- (C - \lambda_0) \int_{D_\eps} (1- \phi^2) \tilde{v}^2 + \CO(\exp)  \nonumber \\
&\leq \, \left( - \frac12 \lambda_0+ \CO({\rho_\eps}^2)\right) \| \tilde{v} \|_{L^2(D_\eps)}^2 
+ \CO(\exp).
\end{align}

Note that $\phi \tilde{v}$  is not
exactly orthogonal to~$U^\prime$. Hence, we cannot directly apply Theorem~\ref{thm:gapAC} in order to control the quadratic form $\inner{L \phi \tilde{v}}{\phi \tilde{v}}_{L^2(\R)}$.  Since the error is only exponentially small, we used that by an easy perturbation argument
\[
\inner{L\phi \tilde{v}}{\phi \tilde{v}} \, = \, -\int_{D_\eps} \left(( \phi \tilde{v})_x \right)^2 
+ \int_{D_\eps} f'(U) ( \phi \tilde{v})^2 
\, \leq \, - \frac12 \lambda_0 \int_{D_{\eps}} \phi^2 \tilde{v}^2 + \CO(\exp). 
\] 
Now, we observe that
\[
\| \tilde{v} \|_{L^2(D_\eps)}^2 \, = \, \int_{D_\eps} v(\eps x)^2 \, \mathrm{d}x \, = \, \eps^{-1} \int_{B_{\delta_\eps}} v(y)^2 \, \mathrm{d}y
\, = \, \eps^{-1} \| v \|^2_{L^2(B_{\delta_\eps})},
\]
and therefore, combining~\eqref{proof:eqAC1} and~\eqref{proof:eqAC2} yields
\begin{equation*}
\inner{\CL^h v}{v}_{L^2(B_{\delta_\eps})} \, \leq \, \left( - \frac12 \lambda_0+ \CO({\rho_\eps}^2)\right) \| v \|^2_{L^2(B_{\delta_\eps})}  + \CO(\exp) \| v \|^2_{L^2(B_{\delta_\eps})}. \qedhere
\end{equation*}
\end{proof}
As a next step, we analyze the spectral gap in the mass conserving case. For this purpose, we denote by $P$ be the
projection of $L^2$ onto the linear subspace~$L^2_0 = \{f \in L^2 \, : \, \int_0^1 f(x) \, \mathrm{d}x = 0 \}$.
Motivated by
\[
\inner{P \CL^h P v}{v}_{L^2} \, = \, 
\inner{\CL^h Pv}{Pv}_{L^2} \, = \, \inner{\CL^h v}{v}_{L^2_0} 
\quad \text{for} \; v \in L^2_0,
\]
we observe that it is sufficient to consider the same operator $\CL^h$ as for the classical Allen--Cahn equation, but
restricted to $L^2_0$, the linear subspace of $L^2$ containing functions with mean~zero. This constraint leads to a subspace of codimension $1$,
\[
L^2_0 \, = \, \left\{ v \in L^2 : \langle v, 1 \rangle\, = \, 0 \right\} 
\,=\, 1^\perp \,=\, PL^2,
\]
and therefore, we need to control the quadratic form on this subspace.
First, we will formulate the problem in general and only after that
consider the special case for the mass conserving Allen--Cahn equation.
The following theorem deals with establishing a spectral gap on a
subspace of codimension $1$. The following simple argument shows that, 
under a suitable angle condition, the Rayleigh quotient
can be bounded from above. This yields a bound on the spectral gap.

\begin{theorem}[Spectral gap on subspaces]
\label{thm:specgap}
Consider a self-adjoint operator $\CL$ on a Hilbert space $\CH$ with
an orthonormal basis of eigenfunctions
$\CL f_k = \lambda_k f_k$
and assume that
\begin{equation}
\label{eq:EV}
\delta \, \geq \, \lambda_1, \, \ldots \, , \lambda_{N+1} \, \geq \, - \delta \, > \, - \lambda
\,\geq \, \lambda_{N+2} \, \geq \, \ldots
\end{equation}
for some $0 < \delta < \lambda$. For $u \in \CH$, we define 
\begin{equation*}
u^\perp \, \coloneqq \, \Big\{ f \in \CH \,  : \, \langle f, u \rangle \, = \, 0 \Big\} \quad \text{and} \quad 
F_u \, \coloneqq \, \frac{1}{\langle f_{N+1},u \rangle} 
\sum_{i=1}^N \langle f_i,u \rangle  f_i + f_{N+1}.
\end{equation*}
Then,
\begin{enumerate}[i)]
\item there exists an $N$-dimensional subspace $\CU$ of $u^\perp$ such 
that
\[
\vert \inner{\CL h}{h} \vert \, \leq \,  \delta \|h\|^2 
\quad \forall h \in \CU. 
\]
\item the condition 
$ \left| \cos \arcangle (F_u,u) \right| \geq \sqrt{\delta / \lambda}$
implies that for $h \perp u, f_1, \ldots, f_{N+1}$
\[
\cfrac{\inner{\CL h}{h}}{\|h\|^2} 
\,\leq \,
\cfrac{ \delta - \lambda \cos^2 \arcangle (F_u,u)}
      { \cos^2 \arcangle (F_u,u) +1}.
\]
\end{enumerate}
\end{theorem}

\begin{proof}
First, we construct an $N$-dimensional subspace corresponding to the small eigenvalues in the interval~$[-\delta, \delta ]$.
For $i = 1,\ldots, N$ define
\begin{equation}
\label{eq:newEF}
g_i \, \coloneqq \, f_i +c_i f_{N+1} \quad \text{with} \quad c_i \, \coloneqq \, - \frac{\langle f_i, u \rangle}{\langle f_{N+1},u \rangle}.
\end{equation}
Obviously, we have 
$g_1,\ldots, g_N \in \mathrm{span} \{f_1, \ldots, f_{N+1} \}$
and $g_1,\ldots, g_N \perp u$ by the definition of the constant $c_i$. It is also straightforward to check that the functions $g_i$ span 
an $N$-dimensional space. This yields directly
\[
- \delta \Vert h \Vert^2 \, \leq \, \langle \CL h,h \rangle \, \leq \, \delta \Vert h \Vert^2 \qquad \text{for} \;\; h \in \mathrm{span} \{ g_1,\ldots, g_N \} \, \eqqcolon \, \CU .
\]
Define
$
V  :=  \mathrm{span} \{ g_1, \, \ldots \,, \, g_N \}^\perp \cap u^\perp
=  \mathrm{span} \{ u, \, g_1, \, \ldots, \, g_N \}^\perp.
$
For $h \in V$ we can then write 
\begin{equation}
\label{eq:LinCom}
h \, = \, \sum_{i=1}^{N+1} \alpha_i f_i + r, \qquad \text{with} \quad r \perp f_i 
\quad \forall i \,=\, 1, \ldots, N+1.
\end{equation}
We have $r,h \perp g_j$ for any $j=1, \ldots, N$ and thereby
\begin{equation}
\label{eq:kernel}
\sum_{i=1}^{N+1} \alpha_i \, \langle f_i, g_j \rangle \, = \, 0.
\end{equation}
With~\eqref{eq:newEF} and $f_i \perp f_j$ for $i \neq j$, we easily 
compute that 
\[
\langle f_i, g_j \rangle_{i,j } \, = \, 
\begin{pmatrix}
1 & 0 & \cdots & 0 & c_1 \\
0 & \ddots & \vdots & 0 & \vdots \\
\vdots & \vdots & \ddots & 0 & \vdots \\
0 & \cdots & \cdots & 1 & c_N \\
\end{pmatrix} 
\in \R^{N \times (N+1)}.
\]
The kernel of this matrix is one-dimensional and spanned by a vector
$\beta \in \R^{N+1}$ with~$\beta_i = - c_i$ for $1 \leq i \leq N$ and
$\beta_{N+1} = 1$. 
By~\eqref{eq:kernel}, $\alpha$ lies in the kernel and we can rewrite~\eqref{eq:LinCom} as
\[
h \, = \, \gamma \sum_{i=1}^{N+1} \beta_i f_i + r \, = \, \gamma \cdot F_u + r, 
\quad \gamma \in \R.
\]
Since $h \in V \subset u^\perp$, we have
$
0  =  \langle h,u \rangle  = 
\gamma \langle F_u,u \rangle + \langle r,u \rangle.
$
This implies immediately that
\[
\gamma^2 \, = \, \frac{\langle r,u \rangle^2}{\langle F_u,u \rangle^2}
\, \leq \, \frac{\Vert r \Vert^2 \Vert u \Vert^2}{\langle F_u,u \rangle^2}.
\]
Thus, we compute
\begin{equation}
\begin{split}
\label{eq:rayleigh}
\cfrac{\langle \CL h,h \rangle}{\Vert h \Vert^2} 
\,&=\, \cfrac{\sum_{j=1}^{N+1} \alpha_j^2 \langle \CL f_j,f_j \rangle 
        + \langle \CL r,r \rangle }{\gamma^2 \Vert F_u \Vert^2 + 
          \Vert r \Vert^2} 
\, \leq \, \cfrac{\delta \sum \alpha_j^2 - \lambda \Vert r \Vert^2}
           {\gamma^2 \Vert F_u \Vert^2 + \Vert r \Vert^2}  \\
&\leq \, \cfrac{\delta \gamma^2 \sum \beta_j^2 - \lambda \Vert r \Vert^2}
           {\gamma^2 \Vert F_u \Vert^2 + \Vert r \Vert^2} 
\, = \, \cfrac{\delta \gamma^2 \Vert F_u \Vert^2 - \lambda \Vert r \Vert^2}
        {\gamma^2 \Vert F_u \Vert^2 + \Vert r \Vert^2}  
\, \leq \, \cfrac{\left( \delta \dfrac{\Vert u \Vert^2 \Vert F_u \Vert^2}{\langle F_u,u \rangle^2} - \lambda \right) \Vert r \Vert^2}
{\gamma^2 \Vert F_u \Vert^2 + \Vert r \Vert^2}.      
\end{split}              
\end{equation}
At this point, we need the angle condition 
\[
\cos \arcangle(F_u,u) 
\, = \, \cfrac{\langle F_u,u \rangle} {\Vert F_u \Vert \Vert u \Vert }
\, \geq \, \sqrt{\delta / \lambda}
\]
to guarantee that the numerator is negative.
Under this assumption, we can continue estimating~\eqref{eq:rayleigh} and derive
\begin{equation*}
\cfrac{\langle \CL h,h \rangle}{\Vert h \Vert^2} 
\, \leq \,
\cfrac{  \delta \dfrac{\Vert u \Vert^2 \Vert F_u \Vert^2}
{\langle F_u,u \rangle^2} - \lambda }
{\dfrac{\Vert u \Vert^2 \Vert F_u \Vert^2}{\langle F_u,u \rangle^2} + 1}
\, = \, \cfrac{ \delta - \lambda \cos^2 \arcangle(F_u,u)}
{1 + \cos^2 \arcangle(F_u,u)}. \qedhere
\end{equation*}
\end{proof} 
Finally, we can apply Theorem~\ref{thm:specgap} to analyze the spectrum of the linearized mass conserving
Allen--Cahn operator. Recall that it is crucial to have a good negative upper bound of the quadratic form orthogonal to the 
tangent space.
We show that in this case the spectral gap is of order $\eps$. This is quite different to the spectral gap for the Allen--Cahn equation~\eqref{eq:AC} without the mass constraint,  which 
by~Theorem~\ref{thm:gapAC} is of order~$1$.
\begin{theorem}[Spectral gap for~\eqref{eq:mAC}]
\label{thm:gapMAC}
Let $v \in L^2_0(0,1)$ with $v \perp u^\xi_i$ for $i=1, \ldots, N$.
Then, we have
\[
\langle \CL^\xi v, v \rangle \, \leq \, 
\Big( - \lambda_0 \eps + \CO(\exp) \Big) \Vert v \Vert^2,
\]
where $\lambda_0$ is the same constant as in Theorem~\ref{thm:gapAC}.
\end{theorem}
\begin{proof}
In the notation of Theorem~\ref{thm:specgap}, we take $u=1 \in L^2(0,1)$ such that
$L^2_0 = \mathrm{span}\{u\}^\perp$, and $f_i = u^h_i$.
Furthermore, we compute 
\[
\inner{f_i}{1}  
\,=\, \int_0^1 u^h_i(x) \, \mathrm{d}x
\,=\, \CO(\exp) + \int_0^1 U^\prime(x \, ; \, h_i, (-1)^{i+1}) \,\mathrm{d}x  
\,=\, 2 (-1)^i + \CO(\exp).
\]
With $F_u$ defined as before in Theorem~\ref{thm:specgap}, this yields
\[
\inner{F_u}{u} \, = \, \dfrac{1}{\inner{ f_{N+1}}{1}} 
\sum_{j=1}^{N+1} \inner{f_j}{1}^2 \, = \, 2(N+1)(-1)^{N+1} + \CO(\exp).
\]
We have $\|f_i\| = \CO(\eps^{-1/2})$ by Proposition~\ref{prop:uh} and 
thus
$
\Vert F_u \Vert  =  (N+1) \cdot \CO(\eps^{-1/2}).
$
Combined we obtain
\[
\cos \arcangle (F_u,u) \, = \, \CO(\eps^{1/2}).
\]
By Proposition~\ref{prop:uh} we have $\CL(u^h) = \CO(\exp)$ and thus differentiating with respect to $h_i$ leads to
${\CL^h u^h_i \, = \, \CO(\exp)}$.
Hence, the first $N+1$ eigenvalues are exponentially small. This shows that we can choose~$\delta = \CO(\exp)$.
Plugging this observation into Theorem~\ref{thm:specgap} yields
\begin{equation*}
\dfrac{\langle \CL^\xi v,v \rangle}{\Vert v \Vert^2} 
\, \leq \,  \dfrac{\delta - \lambda_0 \eps}{1+\eps}
\, \leq \, - \lambda_0 \eps + \CO(\exp). \qedhere
\end{equation*}
\end{proof}
For the classical Allen--Cahn equation we established a spectral gap of order $1$, whereas due to mass conservation the gap
shrinks to $\CO(\eps)$ for~\eqref{eq:mAC}. As we will see later in Theorems~\ref{thm:multiAC:L4stabAC} and~\ref{thm:multiAC:L4mac},
this heavily influences the maximal radius and noise strength that we can treat in our stability analysis.

\section{Analysis of the stochastic ODE along the slow manifold}
\label{sec:AC:ODE}

In this section, we give the stochastic ODEs governing the motion of the kinks for both cases. We show that 
for the non-massconserving Allen--Cahn equation the $N+1$ interfaces
move---up to the time scale where a collision is likely 
to occur---independently according
to Brownian motions projected onto the slow manifold. This is quite different to the 
 mass-conserving case where (as one would expect) the dynamics is coupled through the mass 
constraint.

Before we analyze the stochastic ODEs for the interface motion, we have to introduce a new coordinate
frame, in which we derive the differential equations for the shape variable $h$ and the normal component $v$.
Due to Theorems~\ref{thm:gapAC} and~\ref{thm:gapMAC}, we established good control of the quadratic
form orthogonal to the tangent space~$\CT_h \CM$, or~$\CT_\xi \CM_\mu$, respectively. Therefore, it is fruitful to split the solution to the Allen--Cahn equation into a component on the slow manifold and the orthogonal direction. This leads to the following definition of the Fermi coordinates.

\begin{definition}[Fermi coordinates]
\label{def:Fermi}
Let $u(t)$ be the solution to~\eqref{eq:AC}. For a fixed time $t>0$, we define the pair of coordinates $(h(t),v(t)) \in \Omega_{\rho_\eps} \times L^2(0,1)$
such that
\[
u(t) \, = u^{h(t)} + v(t), \quad v(t) \perp \CT_{h(t)} \CM,
\]
as \textit{Fermi coordinates} of $u(t)$. \\
In case of the mass conserving equation~\eqref{eq:mAC}, the definition works analogously. One only has to replace the set of admissible interface positions $\Omega_{\rho_\eps}$ by the set $\CA_{\rho_\eps}$ (given by~\eqref{def:admissiblemasscon}) and the slow manifold~$\CM$ 
by its mass conserving counterpart $\CM_\mu$.
\end{definition}

Unless we are close to the boundary of the slow manifold, 
for the initial condition $u(0)$ we always find Fermi coordinates 
by considering the point of smallest distance. 
Also note that we do not assume that the Fermi coordinates are uniquely determined,
or that the map $u\mapsto u^h$ is a well defined projection.
Later in Lemma~\ref{lem:lipschitz} and Remark~\ref{rem:lipschitzcoord}, we show that sufficiently close to the slow manifold the Fermi coordinates are at least always defined. We find one possible choice being as smooth as $u$ in time.

For now, we first assume that the coordinate system is well-defined in order to 
derive an equation governing the motion of the kink positions $h$. 
Under the assumption that $h$ performs a diffusion process given by 
\begin{equation}
\label{eq:dh}
dh \, = \, b(h,v) \, dt + \langle \sigma(h,v), dW \rangle,
\end{equation}
one can compute the drift $b$ and diffusion $\sigma$ explicitly by applying It\^o formula to the orthogonality condition of Definition \ref{def:Fermi}.
The computation is straightforward, but quite lengthy. For details see \cite{ABK12} or \cite{SchindlerPhD}. The diffusion term $\sigma$ is 
given by
\begin{equation}
\label{eqAC:sigma}
\sigma_r(h,v) \, = \, \sum_i A_{ri}^{-1} u^h_i,
\end{equation}
and for the drift $b$ we obtain
\begin{equation}
\begin{split}
b_r(h,v) \, = \, &\sum_i A_{ri}^{-1} \langle u^h_i, \CL( u^h + v) \rangle
+ \sum_i A_{ri}^{-1} \sum_j \langle u^h_{ij}, \CQ \sigma_j \rangle \\
\label{eqAC:b}
&+ \sum_{i,j,k} A_{ri}^{-1} \left[ \frac12 \langle u^h_{ijk},v \rangle - \langle u^h_{ij}, u^h_k \rangle 
-\frac12 \langle u^h_i, u^h_{jk} \rangle \right] \langle \CQ \sigma_j, \sigma_k \rangle.
\end{split}
\end{equation}
Note that, for the sake of 
simplicity, we expressed everything with respect to the coordinate $h$, although we 
introduced the coordinate $\xi$ for the mass conserving equation.
An essential point in the computation and analysis of the SDE is the invertibility
of  the matrix $A$  given by
\begin{equation}
\label{ODE:matrix}
A_{kj}(h,v) \, = \, 
\langle u^h_k, u^h_j \rangle - \langle u^h_{kj},v \rangle
\end{equation}
which we will discuss in the next section. 
\begin{remark}
Given a solution $u$, we can replace $v$ in  (\ref{eq:dh}) 
by $u-u^h$, and obtain an equation on the slow manifold for 
the position of the interfaces $h$. 
Note that this equation is not an approximation.
\end{remark}
Moreover, by exactly reverting the calculation that leads to  (\ref{eq:dh}) 
one can verify the following result. See \cite{BS20,SchindlerPhD}.

\begin{proposition}\label{prop:exFerm}
Given a solution $u$ of (AC) or (mAC) let $h$ be a solution 
of \eqref{eq:dh} with $v$ replaced by $u-u^h$, i.e.,
\[
dh \, = \, b(h,u-u^h) \, dt + \langle \sigma(h,u-u^h), dW \rangle,
\]
then $(h,u-u^h)$ are Fermi coordinates for $u$.
\end{proposition}

\subsection{Analysis of the stochastic ODE for (AC)}

We start with the non-massconserving case. Here, we will see that $A$ and its inverse
are diagonal matrices up to terms being small in $\|v\|$.

\begin{lemma}
\label{lem:matrixAC}
For $h \in \Omega_{\rho_\eps}$ consider the matrix $A \in \R^{(N+1) \times (N+1)}$
defined by 
\[
A_{kj} \, \coloneqq \, \inner{u^h_k}{ u^h_j}  - \inner{ u^h_{kj}}{v}.
\]
We obtain
\[
A_{kj} \, = \, 
 \eps^{-1} \left[ \CX + \CO(\eps^{-1/2}) \|v \| \right] \delta_{kj}
+ \CO( \exp ) \|v\|.
\]
Moreover, as long as $\|v\| < c \eps^{1/2 + m}$ for some $m > 0$, 
the inverse $A^{-1}$ is given by
\[
A^{-1} \, = \, \eps\left[ \CX^{-1}  + \CO(\eps^{m}) \right] 
\mathrm{I}_{N+1} 
+ \CO(\exp),
\]
where $\CX$ is the constant given in Proposition~\ref{prop:uh}.
\end{lemma}

\begin{proof}
We obtain
$
\inner{u^h_k}{u^h_j}  =  \CX \eps^{-1} \delta_{kj} {}+{} \CO(\exp)$ and
$\|u^h_{kj}\|  =  \CO(\eps^{-3/2}) \delta_{kj} +  \CO(\exp)
$ by Pro\-position~\ref{prop:uh}.
Thus,  the bound on $A_{kj}$ follows directly by applying the Cauchy--Schwarz inequality.
Using geometric series, this yields for $\| v \|$ sufficiently small
\begin{equation*}
\begin{split}
A^{-1} \, &= \, \left[ \CX \eps^{-1} + \CO(\eps^{-3/2}) \|v\| \right]^{-1}
\mathrm{I}_{N+1} + \CO(\exp) \\
&= \, \CX^{-1} \eps \left[ 1 + \CO(\eps^{-1/2}) \|v \| \right]^{-1} 
\mathrm{I}_{N+1} + \CO(\exp) \\
&= \, \left[ \CX^{-1} \eps + \CO(\eps^{1+m}) \right] \mathrm{I}_{N+1} 
+ \CO(\exp). \qedhere
\end{split}
\end{equation*}
\end{proof}

Before we continue analyzing the stochastic ODE, let us first show that the coordinate frame around $\CM$ given 
by Definition~\ref{def:Fermi} is well-defined.
We prove that, as long as the matrix~$A$ is invertible, i.e., $\|v\| < \eps^{1/2+m}$, and the nonlinearity is bounded, i.e., $v \in L^4$,
 the coefficients$b$ and $\sigma$  defined by~\eqref{eqAC:b} and~\eqref{eqAC:sigma} are Lipschitz continuous with respect to $h$. 
Note that we will only compute the Lipschitz constant  for $\sigma$ explicitly, as we need it for the analysis of the stochastic ODE.

\begin{lemma}[Lipschitz continuity of $b$ and $\sigma$]
\label{lem:lipschitz}
Let $h, \bar{h} \in \Omega_{\rho_\eps}$ and $v \in L^4(0,1)$ satisfying $\|v\| < \eps^{1/2+m}$ for some 
$m>0$.  Then, there exist constants~$C>0$ and $C_\eps > 0$ (depending on $\eps$ and $\|v\|_{L^4}$) such that
\begin{equation}
\label{eq:lipsigma}
\begin{split}
\| \sigma(h,v) - \sigma(\bar{h},v) \| \, \leq \, C \eps^{-1/2} | h - \bar{h} | \quad \text{and} \quad 
\| b(h,v) - b(\bar{h},v) \| \, \leq \, C_\eps | h - \bar{h} |.
\end{split}
\end{equation}
\end{lemma}

\begin{proof}
Note that in the following computation the pair $(h,v)$ does not denote
the Fermi coordinate defined in Definition~\ref{def:Fermi} and therefore, $v$ does not depend on $h$. We start with estimating the derivative of the inverse $A^{-1}(h,v)$.  By construction of $u^h$,
the matrix~$A(h,v)$ is smooth in~$h$ and we compute 
\[
\partial_{h_k} A_{ij} \, = \ \frac{\partial( \inner{u^h_i}{u^h_j} - \inner{u^h_{ij}}{v})}{\partial h_k}
\, = \, \inner{u^h_{ik}}{u^h_j} + \inner{u^h_i}{u^h_{jk}} - \inner{u^h_{ijk}}{v},
\]
which by Proposition~\ref{prop:uh} is exponentially small unless $i=j=k$. In the latter case, we have
\[
\partial_{h_k} A_{kk} \, = \, 2 \inner{u^h_{kk}}{u^h_k} - \inner{u^h_{kkk}}{v} \, = \, \CO(\exp) + \CO(\eps^{-5/2} \|v\|).
\]
By virtue of $\mathrm{D}_hA^{-1}  = - A^{-1} (\mathrm{D}_hA) A^{-1}$ and $A^{-1} = \CO(\eps)$ (cf.~Lemma~\ref{lem:matrixAC}), this yields
\[
\mathrm{D}_hA^{-1} \, = \, \CO(\eps^{-1/2} \|v\|) \, = \, \CO(\eps^m).  
\]
Recall that $\sigma(h,v) = A^{-1} \cdot \partial_h u^h$. Differentiating with respect to $h$ yields
\[
\mathrm{D}_h\sigma \, = \, \mathrm{D}_h A^{-1} \partial_h u^h + A^{-1} \partial_h^2 u^h
\]
and thus, by the previous bound on $\mathrm{D}_h A^{-1}$ and Proposition~\ref{prop:uh},
$
\|\mathrm{D}_h\sigma(h,v) \| \, = \, \CO(\eps^{-1/2}).
$
\\
Since the set $\Omega_{\rho_\eps}$ of admissible interface positions is convex, we have
\[
\sigma(h,v) - \sigma(\bar{h},v) \, = \, \int_0^1 \mathrm{D}_h \sigma(\bar{h} + s(h-\bar{h}),v) \, \mathrm{d}s
\]
and with that we easily obtain~\eqref{eq:lipsigma}.

In order to derive the Lipschitz continuity of $b$, one can analogously verify that $b(h,v)$ is differentiable with respect 
to $h$ and the derivative is bounded. Note that only here we need the condition~$v \in L^4$ to control
the nonlinearity $\inner{ \CN^h(v)}{v}$ appearing in the definition~\eqref{eqAC:b} of $b$.
The careful analysis of the Lipschitz constant can be carried out after
some lengthy calculation. We omit the details here.
\end{proof}

\begin{remark}
\label{rem:lipschitzcoord}
We can use the Lipschitz continuity of the coefficients 
to show that  for a solution~$u$ to~\eqref{eq:AC} the Fermi coordinates given by Definition~\ref{def:Fermi}
are locally well defined.
 Since the multi-kink profiles $u^h$ define smooth functions in $h$, we see that by Lemma~\ref{lem:lipschitz} the maps${h \mapsto b(h, u(t) - u^h)}$ and $h \mapsto \sigma(h, u(t) - u^h)$ are locally Lipschitz continuous in $h$. 
Thus, as long as~$h(t)$ lies in $\Omega_{\rho_\eps}$ and $u-u^h$ is sufficiently small 
(see Lemma \ref{lem:lipschitz}) a unique (local) solution $h(t)$ to~\eqref{eq:dh} with 
$v$ replaced by~$u - u^h$ exists.
The pair $(h,v)$ satisfies the
Definition~\ref{def:Fermi} of the Fermi coordinates. See proposition \ref{prop:exFerm}.
\end{remark}

As the matrix $A$ and its inverse are (up to exponentially small terms) diagonal matrices,  we can show that the stochastic ODE in 
the non-massconserving case essentially decouples fully.
We split equation~\eqref{eq:dh} into its deterministic part and a 
remainder $\CA$, where we collect all terms depending on stochastics, i.e., we write
\begin{equation}
\label{eqAC:dhito}
dh_r \,= \, \sum_i A_{ri}^{-1} \langle u^h_i, \CL( u^h + v) \rangle \, dt
       + d\CA^{(r)},
\end{equation}
where by~\eqref{eqAC:sigma} and~\eqref{eqAC:b}
\begin{equation}
\label{eqAC:dA}
\begin{split}
\sum_j A_{kj} \, d\CA^{(j)} \, = \, &\sum_j \inner{u^h_{kj}}{\CQ\sigma_j} \,dt
+ \sum_{i,j} \left[ \frac12 \inner{u^h_{ijk}}{v} 
- \inner{u^h_{kj}} {u^h_i}  
-\frac12 \inner{u^h_k}{u^h_{ij}} \right] \inner{\CQ \sigma_i}{\sigma_j} 
 \, dt  \\
&+ \inner{u^h_k}{dW}.
\end{split}
\end{equation}
The following lemma deals with estimating the process $\CA$ in terms of $\eps$. We see that---up to lower 
order terms---the $k$-th component of $\CA$ does only depend on the derivatives with respect to $h_k$
and hence it decouples. Moreover, its dominating term is of order $\eta_{\eps}$.

\begin{lemma}
\label{lem:estdA}
As long as $\| v(t) \| < \eps^{1/2 + m}$ for some $m > 0$, we have
\[
d\CA^{(k)} \, = \, A^{-2}_{kk} \inner{u^h_{kk}}{\CQ u^h_k} \, dt + A_{kk}^{-1} \inner{u^h_k}{dW} + \CO(\eps^m \eta_{\eps}) \, dt
+ \inner{\CO_{L^2}(\eps^{1/2 + m})}{dW}.
\]
Moreover, the dominating term can be estimated by 
\[
 A^{-2}_{kk} \inner{u^h_{kk}}{\CQ u^h_k} \, dt + A_{kk}^{-1} \inner{u^h_k}{dW} \, = \, \CO( \eta_{\eps}) \, dt + \inner{\CO_{L^2} (\eps^{1/2})}{dW}.
\]
\end{lemma}
\begin{proof}
Lemma~\ref{lem:matrixAC} and~\eqref{eqAC:sigma} imply directly that
\[
\sigma_r(h,v) \, = \, \left[ \CX^{-1}\eps + \CO(\eps^{1+m}) \right] u^h_r
+ \CO(\exp).
\]
With $\|u^h_r \| = \CO(\eps^{-1/2})$ (cf.~Proposition~\ref{prop:uh}), this yields
$
\|\sigma_r \| \, = \, \CO(\eps^{1/2}).
$
The Cauchy--Schwarz inequality implies for the remaining terms 
of \eqref{eqAC:dA}
\[
\vert \inner{u^h_{kj}}{\CQ \sigma_j} \vert \, \leq \, \|u^h_{kj}\| \,
\|\CQ \| \, \| \sigma_j \| 
\, \leq \, C \eps^{-3/2} \eta_{\eps} \eps^{1/2}
\, = \, C \eps^{-1} \eta_{\eps}
\]
and
\[
\vert \langle u^h_{ijk}, v \rangle \langle \CQ \sigma_i, \sigma_j \rangle \vert
\, \leq \, C \eps^{-5/2} \Vert v \Vert \eta_{\eps} \eps^{1/2} \eps^{1/2}
\, = \, C \eps^{-3/2} \eta_{\eps} \, \Vert v \Vert.
\]
Moreover by Proposition~\ref{prop:uh}, the terms involving inner products of first
and second derivatives of~$u^h$ are exponentially small.
Plugging these estimates into~\eqref{eqAC:dA} yields
\[
\sum_j A_{kj} \, d\CA^{(j)} \, = \, 
\inner{u^h_{kk}}{\CQ \sigma_k} \, dt + \inner{ u^h_k}{dW} 
+ \CO( \eps^{-1+m} \eta_{\eps}) + \CO(\exp).
\]
By using Lemma~\ref{lem:matrixAC}, we obtain 
\begin{equation*}
\begin{split}
d\CA^{(k)} \, &= \, A^{-1}_{kk} \left[ \inner{u^h_{kk}}{\CQ \sigma_k} \, dt
 + \inner{u^h_k + \CO(\exp)}{ dW} 
+ \CO( \eps^{-1+m} \eta_{\eps}) \, dt + \CO(\exp) \, dt \right]  \\
&= \, A^{-2}_{kk} \inner{u^h_{kk}}{\CQ u^h_k}  \, dt
+ A^{-1}_{kk} \inner{ u^h_k + \CO(\exp)}{ dW} 
+ \CO( \eps^{m} \eta_{\eps}) \,dt + \CO(\exp) \, dt. \qedhere
\end{split}
\end{equation*}
\end{proof}

As a next step, we investigate the deterministic part. As we cannot control the
nonlinearity in terms of the~$L^2$-norm, we additionally assume smallness of the normal component $v$ in $L^4$.
In the stability result of Section~\ref{subsec:L4AC}, the maximal $L^4$-radius that we can treat is of 
order~$\eps^{1/4 + m/2-\kappa}$ for small~${\kappa > 0}$.

\begin{lemma}
\label{lem:ACestODE}
Let $m > 0$ and $\kappa > 0$ be very small. For $h\in \Omega_{\rho_\eps}$ and $v \perp u^h_i, i=1,\ldots, N+1$, assume that
$\| v \| < \eps^{1/2 + m}$ and 
$\| v \|_{L^4} < \eps^{1/4 + m/2 - \kappa}$.
Then, we have
\[
\sum_i A_{ri}^{-1} \inner{ u^h_i}{ \CL( u^h + v)}
\, \leq \, C \eps^{2m + 1 - 2 \kappa}.
\]
\end{lemma}

\begin{proof}
Expanding $\CL$ yields $\CL(u^h + v) \, = \, \CL(u^h) + \CL^h v + \CN^h(v)$.
We observe that ${\CL(u^h) = \CO(\exp)}$ by Proposition~\ref{prop:uh}.
Differentiating with respect to $h_i$ yields $\CL^h u^h_i = \CO(\exp)$ and hence,
$\langle \CL^h v, u^h_i \rangle = \langle v, \CL^h [ u^h_i ] \rangle = \CO(\exp)$, since $\CL^h$ is self-adjoint.
The remaining nonlinear term is estimated by
\begin{equation*}
\begin{split}
\inner{\CN^h(v)}{ u^h_i} \, &= \, \int_0^1 3 u^h u^h_i v^2 - u^h_i v^3 \, 
\mathrm{d}x 
\leq \, C \eps^{-1} \left[ \| v \|^2 + \| v \|_{L^3}^3
\right] \\
&\leq \, C \eps^{-1} \left[ \| v \|^2 +
\| v \|\| v \|_{L^4}^2 \right] 
\leq \, C \eps^{2m -2\kappa},
\end{split}
\end{equation*}
where we interpolated the $L^3$-term by Hölder's inequality.
Applying Lemma~\ref{lem:matrixAC} concludes the proof.
\end{proof}

We can finally show that, up to times of order $\CO(\eps \eta_{\eps}^{-1})$, the motion of the 
kinks is approximately given by the projection of the Wiener process onto the slow
manifold $\CM$, that is, for~$k=1,\ldots,N+1$ 
\begin{equation}
\label{eq:approxdh}
d\tilde{h}_k \, = \, \frac{1}{\| u^{\tilde{h}}_k \|^2} \inner{u^{\tilde{h}}_k}{\circ \, dW}.
\end{equation}
At times of order $\CO(\eps \eta_{\eps}^{-1})$ the droplet is expected to move by the magnitude of $\eps$
and hence, we treat the relevant time scale in our analysis, since we have to assure that the distance between two
kinks is at least $\eps^{1-}$ (cf.~Remark~\ref{rem:timescale}).
For a sufficiently large noise strength, the stochastic effects dominate the dynamics and hence,
as expected, the approximation by the purely stochastic process is better for a larger noise strength $\eta_{\eps}$. In our main stability
result (see Theorem~\ref{thm:multiAC:L4stabAC}) the maximal strength we can treat is of order $\eps^{1 + 2m}$. 

\begin{theorem}[Approximation of the exact dynamics]
\label{thm:approxBB}
Let $h(t)$ be a solution to~\eqref{eq:dh} and $\tilde{h}(t)$ be a solution to~\eqref{eq:approxdh}.
For~${m> 0}$ and small $\kappa > 0$, define the stopping time
\begin{equation*}
\tau \, \coloneqq \, \inf \left\{ t \geq 0 \; : \; h(t) \notin \Omega_{\rho_\eps} \quad \text{or}  \quad\| v \| \, > \, \eps^{1/2 + m} \quad \text{or} \quad
\| v \|_{L^4} \, > \, \eps^{1/4 + m/2 - \kappa} \right\}.
\end{equation*}
Then, for a stopping time $T \leq c \eps \eta_{\eps}^{-1} \wedge \tau$, we obtain
\begin{align*}
\E \sup_{0 \leq t \leq T} \vert h(t) - \tilde{h}(t) \vert \, \leq \, C \eps +  C\eps^{2m+ 2 - 2\kappa}  \eta_{\eps}^{-1}.
\end{align*}
\end{theorem}

\begin{proof}
For notational convenience, we define for $h, \tilde{h} \in \Omega_{\rho_\eps}$ the maps
\[
\gamma_r(h) \, \coloneqq \, \frac{u^h_r}{ \| u^h_r\|^2} \quad \text{and} \quad \Delta(h,\tilde{h}) \, \coloneqq \, \gamma_r(h) - \gamma_r(\tilde{h}).
\]
By~\eqref{eq:dh}, \eqref{eqAC:sigma}, and Lemma~\ref{lem:matrixAC}, we derive
for $t \leq T$ 
\begin{equation*}
 h_r(t) - \tilde{h}_r (t) \, \leq \, \int_0^t  b_r(s) + I_r(\tilde{h}(s)) \, \mathrm{d}s + 
 \int_0^t \inner{\Delta(h,\tilde{h}) +\CO(\eps^{1+m}) u^h_r\, }{ \, \mathrm{d}W}.
\end{equation*}
Here, $I(\tilde{h})$ collects all the terms that appear after a conversion of the Stratonovich SDE~\eqref{eq:approxdh} into an It\^o SDE.
This is important, as we need the stochastic integral to be a martingale. These It\^o-Stratonovich correction terms are essentially identical to the terms
in~\eqref{eqAC:b}, where we set $v=0$ and replace the matrix $A(\tilde{h},v)$  by $S_{kj}(\tilde{h}) = A(\tilde{h},0) = \langle u_k^{\tilde{h}}, u_j^{\tilde{h}} \rangle$. In more detail, one easily computes that
\begin{equation*}
\begin{split}
I_r(\tilde{h}) \, &= \, \sum_i S_{ri}^{-1} \sum_j \langle u^h_{ij}, \CQ \sigma_j(\tilde{h},0) \rangle + \sum_{i,j,k} S_{ri}^{-1} \left[ - \langle u^h_{ij}, u^h_k \rangle 
-\frac12 \langle u^h_i, u^h_{jk} \rangle \right] \langle \CQ \sigma_j(\tilde{h},0), \sigma_k(\tilde{h},0) \rangle \\
&= \, S_{rr}^{-2} \langle u^{\tilde{h}}_{rr}, \CQ u^{\tilde{h}}_r \rangle + \CO(\exp)  = \CO(\eta_{\eps}).
\end{split}
\end{equation*}
Here, we utilized that $S_{ri} = \langle u_r^{\tilde{h}}, u_i^{\tilde{h}} \rangle = \CX^{-1} \eps \delta_{ri} + \CO(\exp)$  by  Proposition~\ref{prop:uh} and, as $v=0$,  $\sigma_r(\tilde{h},0) = \sum S_{ri}^{-1} u^{\tilde{h}}_i = S_{rr}^{-1} u^{\tilde{h}}_r + \CO(\exp)$. Moreover, the inner product of first derivatives with 
second derivatives of $u^h$ is exponentially small due to Proposition~\ref{prop:uh}. \\
In Lemmata~\ref{lem:estdA} and~\ref{lem:ACestODE},  we established an $L^\infty$-bound for $b$ up to the stopping time $\tau$, namely,
\[
\sup_{0 \leq t \leq \tau} \vert b \vert \, \leq \, c (\eta_{\eps} + \eps^{2m+1 - 2 \kappa} )
\]
Combining this with the bound of the Ito-Stratonovich correction term $I$ yields
\begin{equation*}
\begin{split}
\E \sup_{0 \leq t \leq T} \vert h_r(t) - \tilde{h}_r(t) \vert \, \leq \, c (\eta_{\eps} + \eps^{2m+1 - 2 \kappa} )  T 
 + \E \sup_{0 \leq t \leq T} \Big\vert \int_0^t \inner{\Delta(h,\tilde{h}) +\CO(\eps^{1+m}) u^h_r}{\mathrm{d}W} \Big\vert.
\end{split}
\end{equation*}
By Burkholder's inequality and Lipschitz continuity of $\gamma$ with Lipschitz constant of order~$\CO(\eps^{-1/2})$ 
(cf.~Lemma~\ref{lem:lipschitz}), the martingale term is estimated by
\begin{align*}
\E &\sup_{0 \leq t \leq T} \Big\vert \int_0^t \inner{\Delta(h,\tilde{h})+\CO(\eps^{1+m}) u^h_r}{\mathrm{d}W} \Big\vert \\
&\leq \, C \E \left[ \int_0^T \inner{\Delta(h,\tilde{h}) +\CO(\eps^{1+m}) \, u^h_r}
{\CQ(\Delta(h,\tilde{h}) +\CO(\eps^{1+m}) \, u^h_r)} \, \mathrm{d}s \right]^{1/2} \\
&= \, C \E  \left[ \int_0^T \inner{\Delta(h,\tilde{h})}{\CQ\Delta(h,\tilde{h}) }
                    +\CO(\eps^{1+m}) \inner{\CQ u^h_r}  {\Delta(h,\tilde{h})}
                    +\CO(\eps^{2+2m}) \inner{ u^h_r}{\CQ u^h_r} \right]^{1/2} \\
&\leq \, C \E \left[ \int_0^T \eta_{\eps} \|\Delta(h,\tilde{h}) \|^2
                    +\eps^{1/2+m} \eta_{\eps} \| \Delta(h,\tilde{h}) \|
                    +\eps^{1+2m} \eta_{\eps} \right]^{1/2} \\
&\leq \,  C \E \left[ \int_0^T \eta_{\eps} \|\Delta(h,\tilde{h})\|^2     
                                                    +\eps^{1+2m} \eta_{\eps} \right]^{1/2} \\
&\leq \,  C \E \left[ \int_0^T \eta_{\eps} \eps^{-1} \|h(s) - \tilde{h}(s)\|^2 \, \mathrm{d}s    
                                                    +\eps^{1+2m} \eta_{\eps} T \right]^{1/2}  \\
&\leq \,  C \eta_{\eps}^{1/2} \eps^{-1/2} T^{1/2} \, \E \sup_{0\leq t \leq T} \|h(t) - \tilde{h}(t)\|^2 
                                               +C\eps^{1/2+m} \eta_{\eps}^{1/2} T^{1/2}.                                               
\end{align*}
With the assumption $T < c\eps \eta_{\eps}^{-1}$, this implies 
\begin{equation*}
\begin{split}
\E \sup_{0 \leq t \leq T} \vert h(t) - \tilde{h}(t) \vert \, \leq \, C \frac{(\eta_{\eps} + \eps^{2m+1 - 2 \kappa} )T 
+ \eps^{1/2+m} \eta_{\eps}^{1/2} T^{1/2} }
{1 - \eta_{\eps}^{1/2} \eps^{-1/2} T^{1/2}} 
\, \leq \, C \eps +  C\eps^{2m+ 2 - 2\kappa}  \eta_{\eps}^{-1}. \qedhere
\end{split}
\end{equation*}
\end{proof}

\begin{remark}
\label{rem:timescale}
In the definition of admissible parameters $\Omega_{\rho_\eps}$, we had to assume that the distance between
two interfaces is bounded from below by $\CO(\eps^{1-})$. Since 
${\E \, h(\eps \eta_{\eps}^{-1}) \, = \, h(0) + \CO(\eps)}$,
the~interface positions $h(t)$ might have moved by order $\eps$ and thus, a collision of two interfaces can occur, which we
cannot treat in our analysis. Therefore, up to the relevant time, the motion of the kinks behaves
approximately like a Wiener process projected onto the slow manifold. After a breakdown of two interfaces, we could restart our analysis on 
a lower-dimensional slow manifold, where the number of kinks is reduced by two, or one if a kink is annihilated at the boundary. We do not cover this annihilation here.
In the case when our analysis breaks down at the boundary of the slow manifold, we are still too far away from the one with less kinks.
\end{remark}
\subsection{\texorpdfstring{Analysis of the stochastic ODE for \eqref{eq:mAC}}{Analysis of the stochastic ODE for (mAC)} }
To conclude our study of the kink motion, we analyze the mass conserving Allen--Cahn equation. Recall that in 
this case, due to mass conservation, we reduced the parameter space $\Omega_{\rho_\eps}$
via~$h_{N+1}(h_1,\ldots,h_N)$  by one dimension and therefore obtain by chain rule and Lemma~\ref{lem:implfct}
\begin{equation}
\label{eq:multiAC:chainrule}
u^\xi_k \, = \, u^h_k + (-1)^{N-k} u^h_{N+1} + \CO(\exp).
\end{equation}

\begin{remark}
Analogously to Remark~\ref{rem:lipschitzcoord}, we can verify that the Fermi coordinates~$(\xi, v)$ around $\CM_\mu$ are locally well-defined
(cf.~Definition~\ref{def:Fermi}).
The crucial point is that the maps $\xi \mapsto b(\xi, u - u^\xi)$ and~$\xi \mapsto \sigma(\xi, u-u^\xi)$ are sufficiently smooth.
 In Lemma~\ref{lem:lipschitz}, we proved the local Lipschitz continuity of the corresponding maps in the non-massconserving case. In fact, let us show that these maps are even smoother. 
By the expressions in~\eqref{eqAC:sigma} and~\eqref{eqAC:b}, the coefficients $\sigma$ and $b$ depend on $\xi$ via various derivatives of $u^\xi$ (up to the
third order). Note that also the matrix $A$ only depends on derivatives of $u^\xi$. Hence, if the profiles $u^\xi$ are sufficiently smooth, the smoothness is directly 
inherited to the coefficients of the stochastic ODE and we then obtain a unique local solution to $d\xi = b(\xi, u - u^\xi) \, dt + \langle
\sigma(\xi, u-u^\xi), dW \rangle$. 

In our construction of the slow manifold, we summed up rescaled and translated solutions to the ODE $U'' - F'(U) = 0$.
In the toy case $F(u) = \tfrac14 (u^2-1)^2$, one obtains the explicit solution~$\tanh(x / \sqrt2)$, which is of course $C^\infty$-smooth. 
Thereby, we see that the multi-kink configuration~$u^\xi$ is sufficiently smooth with
respect to $\xi$, which shows that the aforementioned maps are at least $C^1$-functions. For details on how to obtain the well-definedness of the
Fermi coordinates, we refer to Remark~\ref{rem:lipschitzcoord}.
\end{remark}

Just like in the analysis of~\eqref{eq:AC}, we first show the invertibiliy of the matrix $A$.
To start with, we consider the metric tensor $S_{kj}  =  \inner{ u^\xi_k}{ u^\xi_j},$
which does not depend on $v$. Due to the coupling through the
mass constraint, the matrix $S$ and its inverse are no longer diagonal. As we will
see, this has an impact on the stochastic ODE governing the motion of
the kinks.

\begin{lemma}
\label{lem:maininverse}
For $u^\xi \in \CM_\mu$ and $j,k \in \{1,\ldots,N\}$ we have
\[
S_{kj} \, = \, \inner{ u^\xi_k}{ u^\xi_j} \, = \, \CX \eps^{-1} \left[ \delta_{kj} + (-1)^{k+j} \right]
+ \CO(\exp),
\]
where $\CX$ is the constant given in Proposition~\ref{prop:uh}.
\end{lemma}

\begin{proof}
With Proposition~\ref{prop:uh} and the chain rule~\eqref{eq:multiAC:chainrule}, we compute
\begin{equation*}
\begin{split}
\langle u^\xi_k, u^\xi_j \rangle 
\, &= \, \langle u^h_k + (-1)^{N-k} u^h_{N+1}, u^h_j + (-1)^{N-j} u^h_{N+1} \rangle \\
&= \, \Vert u^h_k \Vert^2 \, \delta_{jk} + (-1)^{k+j} \Vert u^h_{N+1} \Vert^2
+ \CO(\exp) 
\, = \, \CX \eps^{-1} \left[ \delta_{kj} + (-1)^{k+j} \right]
+ \CO(\exp). \qedhere
\end{split}
\end{equation*}
\end{proof}

With the structure of the matrix at hand, we can easily invert $S$.

\begin{lemma}
\label{lem:invertMC}
Let $u^\xi \in \CM_\mu$. The matrix $S$ is invertible with 
\[
S^{-1}_{kj} \, = \, \eps/ \CX \left[ \delta_{kj} + \frac{1}{N+1} (-1)^{k+j+1}
\right] + \CO(\exp).
\]
\end{lemma}

\begin{proof}
We have (ignoring exponentially small terms)
\begin{equation*}
\begin{split}
\sum_{j=1}^N S_{kj} S^{-1}_{jl} \, &= \, 
\sum_{j=1}^N \left[ \delta_{kj} + (-1)^{k+j} \right]
\left[ \delta_{jl} + \frac{1}{N+1} (-1)^{j+l+1}
\right] \\
&= \, \delta_{jl} + \frac{1}{N+1} (-1)^{k+l+1} + (-1)^{k+l} 
+ \frac{N}{N+1} (-1)^{k+l+1} 
\, = \, \delta_{kl}. \qedhere
\end{split}
\end{equation*}
\end{proof}

Finally, we show that---as long as $\Vert v \Vert$ stays sufficiently small---the full matrix $A(\xi,v)$ given by~\eqref{ODE:matrix}
is invertible. With that, the coefficients of the It\^o diffusion~\eqref{eq:dh} (with $h$ replaced by~$\xi$) are well-defined and we can continue
to study the dynamics of kinks for the mass conserving Allen--Cahn equation in more detail.

\begin{lemma}
\label{lem:multiAC:inversemAC}
Consider the matrix $A_{kj}(\xi,v)  =  S_{kj} - \inner{u^\xi_{kj}}{v}$, where $S$ is given by Lemma~\ref{lem:maininverse}.
\\
Then, as long as $\Vert v \Vert <  \eps^{1/2 + m}$ for some 
$m > 0$, $A$ is invertible with
\[
A^{-1} \, = \, S^{-1} + \CO( \eps^{m + 1}).
\]
\end{lemma}

\begin{proof}
For a small perturbation $S(v)$, given by $S_{kj}(v) = \inner{u^\xi_{kj}}{v}$,  of the matrix $S$ we compute via geometric series
\begin{equation*}
\begin{split}
A^{-1} \, &= \, \left[ S - S(v) \right]^{-1} \, = \,  \left[ \mathrm{I}_N - S^{-1} S(v) \right]^{-1} S^{-1} \\
&= \, \sum_{j \in \N} \left[ S^{-1} S(v) \right]^j S^{-1} 
\, = \, S^{-1} + \sum_{j=1}^\infty \left[ S^{-1} S(v) \right]^j S^{-1}
\, = \, S^{-1} + \CO(\eps^{m + 1}),
\end{split}
\end{equation*}\\[-0.6em]
where we used that $S(v) = \CO(\eps^{-3/2} \Vert v \Vert)$ and the sum converges for $\Vert v \Vert < \eps^{1/2 + m}$.
\end{proof}

We continue with estimating the deterministic part of~\eqref{eq:dh}. Similarly to Lemma~\ref{lem:ACestODE}, we 
have to assume smallness of the normal component $v$ in $L^2$ and $L^4$ to control the nonlinearity. In the following lemma, we consider the radii for which we show stochastic stability 
later in Sections~\ref{subsec:L2mAC} and~\ref{subsec:L4mAC}.

\begin{lemma}
\label{lem:estODE}
Let $m > 0$, $\xi \in \CA_{\rho_\eps}$, and $v \perp \CT_{u^\xi} \CM_\mu $. Also, assume that
$\| v \| < \eps^{3/2 + m}$ and~
$\| v \|_{L^4} < \eps^{3/4 + m/2- \kappa}$.
Then, we obtain
\[
\inner{ u^\xi_i}{ \CL( u^\xi + v)}
\, \leq \, C \eps^{2 + 2m -2 \kappa}.
\]
\end{lemma}

\begin{proof}
We follow the proof of Lemma~\ref{lem:ACestODE}.
Only for the nonlinearity we have to take the different radii into account. By Hölder's inequality we obtain
\begin{equation*}
\inner{\CN^\xi(v)}{ u^\xi_i}  
\, \leq \, C \eps^{-1} \left[ \| v \|^2 + \| v \|_{L^3}^3
\right] 
\, \leq \, C \eps^{-1} \left[ \| v \|^2 +
\| v \| \| v \|_{L^4}^2 \right] 
\, \leq \, C \eps^{2 + 2m -2 \kappa}. \qedhere
\end{equation*}
\end{proof}

In order to analyze the SDE governing the motion of kinks, it is more convenient to rewrite~\eqref{eq:dh}
in the Stratonovich sense. 
By leaving out It\^o corrections, Lemmata~\ref{lem:invertMC} and~\ref{lem:multiAC:inversemAC} imply
\begin{equation*}
\begin{split}
d\xi_r \, &= \, \sum_i A_{ri}^{-1} \inner{ \CL(u^\xi + v)}{u_i^\xi}  \, dt
+ \sum_i A_{ri}^{-1} \inner{u^\xi_i}{ \circ \, dW}  \\
&= \, \sum_i S_{ri}^{-1} \inner{ \CL(u^\xi + v)}{u_i^\xi } \, dt
+ \sum_i S_{ri}^{-1} \inner{ u^\xi_i}{ \circ \, dW} 
+ \CO( \eps^{ 4 + 3m } ) \, dt
+ \inner{ \CO_{L^2} ( \eps^{5/2 + m})} { \circ \, dW} \\
&= \, \CX^{-1} \eps \, \inner{ \CL(u^\xi + v)}{u_r^\xi } \, dt +  \CX^{-1} \eps \, \inner{u^\xi_r}{ \circ \, dW} \\
&\quad+ \frac{(-1)^r \eps}{ \CX (N+1)} \, \inner{ \CL(u^\xi + v)}{\sum_{i=1}^{N} (-1)^{i+1} u_i^\xi } \, dt
  + \frac{(-1)^r \eps}{ \CX (N+1)} \inner{\sum_{i=1}^{N} (-1)^{i+1} u_i^\xi}{ \circ \, dW} \\
&\quad+ \CO( \eps^{ 4 + 3m } )\, dt
+ \inner{ \CO_{L^2} ( \eps^{5/2 + m})} { \circ \, dW}.
\end{split}
\end{equation*} 
The first two summands (depending only on $u^\xi_r$) are similar to the non-massconserving case, but---due to
the mass constraint---we obtain additional terms, which do not only depend on the 
position~$\xi_r$ but rather on all positions $( \xi_1, \ldots, \xi_N)$.
To give a better understanding of this equation---especially of the additional terms---let us express it in the original $h$-coordinates. 
Recall that by chain rule $u^\xi_i \, = \, u^h_i + (-1)^{N-i} u^h_{N+1} + \CO(\exp).$
Thus we compute (ignoring exponentially small terms)
\begin{align}
\label{eq:xiandh}
u^\xi_r {}&+{} \frac{(-1)^r }{N+1}  \sum_{i=1}^N (-1)^{i+1} u^\xi_i 
\, = \, u^h_r + (-1)^{N-r} u^h_{N+1} + \frac{(-1)^{r+N+1} N }{N+1}  u^h_{N+1} 
+ \frac{(-1)^r }{N+1} \sum_{i=1}^N (-1)^{i+1}  u^h_i \nonumber  \\
&= \, u^h_r +(-1)^r u^h_{N+1} \left[ \frac{(-1)^{N+1} N }{N+1} - (-1)^{N+1}    \right] 
+ \frac{(-1)^r }{N+1} \sum_{i=1}^{N} (-1)^{i+1}  u^h_i \nonumber  \\
&= \, u^h_r + \frac{(-1)^r}{N+1} (-1)^{N} u^h_{N+1} + \frac{(-1)^r }{N+1} \sum_{i=1}^{N} (-1)^{i+1}  u^h_i 
= \, u^h_r + \frac{(-1)^r }{N+1} \sum_{i=1}^{N+1} (-1)^{i+1}  u^h_i.
\end{align}
Plugging this into the Stratonovich SDE yields
\begin{equation*}
\begin{split}
d\xi_r \, &= \, \|u^h_r\|^{-2} \, \inner{ \CL(u^h + v)}{u^h_r } \, dt +  \|u^h_r\|^{-2}  \, \inner{u^h_r}{ \circ \, dW} \\
&\quad + \frac{(-1)^r }{(N+1)} \sum_{i=1}^{N+1} (-1)^{i+1} \left[ \|u^h_i\|^{-2} \inner{ \CL(u^h + v)}{u^h_i} \, 
  + \frac{(-1)^r }{(N+1)} \|u^h_i\|^{-2} \inner{u^h_i}{ \circ \, dW}\right] \\
&\quad + \CO( \eps^{ 4 + 3m } ) \, dt
+ \inner{ \CO_{L^2} ( \eps^{5/2 + m})} { \circ \, dW}.
\end{split}
\end{equation*}
We observe that all the terms appearing in this formula are up to an exponentially small error the right-hand side of the equation for $dh$
(see~\eqref{eqAC:dhito} and~\eqref{eqAC:dA} with $A(h,v)$ a diagonal matrix).
Thus, we have
\begin{equation}
\label{eq:dxianddh}
d\xi_r \, \approx \, dh_r + \frac{(-1)^r }{(N+1)} \sum_{i=1}^{N+1} (-1)^{i+1} dh_i.
\end{equation}
Therefore, the kink motion for the mass conserving Allen--Cahn equation is approximately given by the independent 
motion of the position $h_r$, which is moving according to the non-massconserving case,                                                                                              
plus a weighted motion of all interface positions $(h_1, \, \ldots \, ,h_{N+1})$ that guarantees the conservation of mass.

\begin{remark}
In Theorem~\ref{thm:approxBB}, we proved that up to times of order $\eps \eta_{\eps}^{-1}$ the interface positions~$h(t)$ behave
approximately like the projection of the Wiener process onto the slow manifold~$\CM$. Using that
$dh_r \, \approx \,  \|u^h_r\|^{-2} \inner{u^h_r}{\circ \, dW}$
and plugging this into~\eqref{eq:dxianddh}, we obtain heuristically
\begin{equation}
\label{eq:multiAC:projectionmAC}
d\xi_r \, \approx \, \sum_{i} S_{ri}^{-1} \inner{u^\xi_i}{\circ \, dW},
\end{equation}
where we essentially used the identity~\eqref{eq:xiandh}. Since the matrix $S$ is given by $S_{ri} = \inner{u^\xi_r}{u^\xi_i}$, we expect that also the dynamics 
for the mass conserving Allen--Cahn equation behaves approximately like the projection of the Wiener process onto $\CM_\mu$.
Analogously to Theorem~\ref{thm:approxBB}, we could make this rigorous and estimate the error for a given time scale.
Opposed to the previous analysis of \eqref{eq:AC}, we cannot quite reach a good error estimate up the relevant time scale of 
order $\CO(\eps \eta_{\eps}^{-1})$, which corresponds to the time that a kink is likely to move by the order of $\eps$ (see Remark~\ref{rem:timescale}).
Basically, this deficiency stems from the worse spectral gap in Theorem~\ref{thm:gapMAC}, which leads to a smaller maximal noise strength that we can 
treat in our stability analysis. See Theorem~\ref{thm:multiAC:L4mac}, where we can allow only for $\eta_{\eps} \leq \eps^{4+2m-\kappa}$, and 
Theorem~\ref{stab:mainest} for the interplay between the spectral gap and the noise strength. 
For a reasonable result, we need that the error, which is linear in the time scale $T_\eps$, is smaller than the magnitude of the process $\xi$, which grows
like~$T_\eps^{1/2}$. 
In the case of the mass-conserving Allen--Cahn equation, we expect the following result to hold true but omit the details. 

\end{remark}

\begin{conj}
Let $\xi(t)$ be the solution to~\eqref{eq:dh} with $b$ and $\sigma$ given by~\eqref{eqAC:b} and~\eqref{eqAC:sigma} and$h$ replaced
by $\xi$. Furthermore, let $ \bar{\xi}(t)$ be the projection of the Wiener process $W$ onto the mass conserving manifold $\CM_\mu$
given by~\eqref{eq:multiAC:projectionmAC}.
For $m > 0$ and small $\kappa > 0$, define the exit time
\[
\tau \, \coloneqq \, \inf \left\{ t \geq 0 \; : \; \xi \notin \CA_{\rho_\eps} \quad \text{or} \quad \| v(t) \| > \eps^{3/2 + m} \quad \text{or} \quad
\| v(t) \|_{L^4} > \eps^{3/4 + m/2 - \kappa} \right\}
\]
Then, for $T_\eps \leq c \eps \eta_{\eps}^{-1} \wedge \tau$, we obtain
\[
\E \sup_{0 \leq t \leq T_\eps} \vert \xi(t) - \bar{\xi}(t) \vert \, \leq \, c \left[ \eta_{\eps} + \eps^{3+2m-2\kappa} \right] T_\eps.
\]
\end{conj}

\begin{remark}
In the proof of the analogous result in Theorem~\ref{thm:approxBB}, it was crucial to explicitly know the Lipschitz constant of the map $\xi \mapsto \sigma(\xi,v)$
provided $v$ is sufficiently small. In establishing the Lipschitz continuity (Lemma~\ref{lem:lipschitz}), we relied on the convexity of the set of admissible
interface positions. While this is straightforward for the set $\Omega_{\rho_\eps}$, this is not quite true in the mass conserving case, where the set of admissible 
positions is given by 
\[
\CA_{\rho_\eps} \, = \, \Big\{ (h_1, \, \ldots \,h_{N},\, h_{N+1}(h_1, \, \ldots \,h_{N}) ) \in \Omega_{\rho_\eps} \, : \, (h_1, \, \ldots \, ,h_N)\in[0,1]^N \Big\}.
\]
 By Lemma~\ref{lem:implfct}, the map $h_{N+1}$ is explicitly given by
\[
h_{N+1}(h_1, \, \ldots \, , h_N) \, = \, \sum_{i=1}^N (-1)^{N-i} h_i + c(\mu) + \CO(\exp),
\]
where we have to introduce a constant $c(\mu)$ depending only on the mass $\mu$. With this expression, one readily computes that
$h_{N+1}( \xi + \lambda( \xi - \bar{\xi}) = h_{N+1}(\xi) + \lambda h_{N+1}( \xi - \bar{\xi}) + \CO(\exp)$ for any~${\xi, \bar{\xi} \in \R^N}$ and~$\lambda \in (0,1)$.
Combined with the convexity of $\Omega_{\rho_\eps}$, this shows that the set $\CA_{\rho_\eps}$ is not exactly convex, but the error is exponentially small.
We obtain the following result:
\[
h, \bar{h} \in \CA_{\rho_\eps} \implies \lambda h + (1-\lambda) \bar{h} \in \CA_{2 {\rho_\eps}} \quad \forall \lambda \in (0,1).
\]
With this property at hand, we expect to bound the Lipschitz constant in the mass conserving case. For some of the technical details, we follow closely 
the proof of  Lemma~\ref{lem:lipschitz}.
\end{remark}

\section{Stochastic Stability}
\label{sec:stab}

In this section, we discuss stochastic stability, both for ~\eqref{eq:AC} and~\eqref{eq:mAC}. The first part is concerned with establishing stability in $L^2$ 
which is crucial for defining the Fermi coordinates (cf.~Definition~\ref{def:Fermi}). Note that this is not sufficient for the analysis of the SDE, where we additionally assumed 
that $v$ is small in~$L^4$ in order to handle the nonlinear terms. Hence, the second part of this section is devoted to stochastic stability in $L^4$. The underlying problem is to prove 
that the normal component $v$ measuring the distance to the slow manifold, 
which satisfies the stochastic PDE
\[
dv \, = \, \left[ \CL(u^h) + \CL^h v + \CN^h(v) \right] dt + dW
- \sum_j u^h_j dh_j - \frac12 \sum_{i,j} u_{ij}^h \inner{\CQ \sigma_i}
{\sigma_j} \, dt,
\]
remains small in various norms. For this purpose, we aim to show a stochastic differential inequality of the type
\begin{equation}
\label{stab:mainest}
d \Vert v \Vert \, = \, - a_\eps \Vert v \Vert \, dt + \CO(K_\eps) \, dt + \langle \CO( c_\eps \Vert v \Vert^\alpha), \, dW_\eps \rangle
\end{equation}
for some positive $\eps$-dependent constants $a_\eps$ and $K_\eps$. With \eqref{stab:mainest} at hand, one can
prove the following theorem which serves as our main tool in the upcoming stability discussions. For more details and a proof, we refer to \cite{ABK12}. Here, this method was used 
for the stability analysis of the one-dimensional stochastic Cahn--Hilliard equation. See also \cite{ABBK15, BS20, SchindlerPhD}.

\begin{theorem}
\label{meta:stability}
Define the stopping time
\[
\tau^\ast \, = \, \inf \left\{ t \in [0, T_\eps \wedge \tau_0 ] \, : \, \Vert v(t) \Vert > R_\eps \right\},
\]
where the deterministic cut-off $T_\eps$ satisfies $T_\eps = \eps^{-M}$ for any fixed large $M>0$ and $\tau_0$ denotes the first exit time
from $\Omega_{\rho_\eps}$, the set of admissible kink positions. Assume that for $ t \leq \tau^\ast$ equation \eqref{stab:mainest} is satisfied with
some positive constants $a_\eps, K_\eps, c_\eps$ and $\alpha$. Furthermore, assume that for some small $\kappa > 0$
\[
\frac{K_\eps + c_\eps^2 \eta_{\eps} R_\eps^{2 \alpha - 1}}{a_\eps R_\eps} \, = \, \CO(\eps^\kappa) \quad \text{and} \quad \Vert v(0) \Vert \, \leq \, \frac{K_\eps + c_\eps^2 \eta_{\eps} R_\eps^{2 \alpha - 1}}{a_\eps}.
\]
Then, the probability $\prob(\tau^\ast < T_\eps \wedge \tau_0)$ is smaller than any power of $\eps$, as $\eps$ tends to zero.
\end{theorem}

\subsection{\texorpdfstring{$L^2$-Stability for \eqref{eq:AC}} {L2-Stability for (AC)}}

We start with the analysis of~\eqref{eq:AC} without mass conservation. Crucial for establishing stochastic stability is the following theorem,
which relies on the spectral gap derived in Theorem~\ref{thm:gapAC}. As long as the $L^2$-norm of the normal component $v$ stays
sufficiently small, the nonlinear term does not destroy the spectral estimate.

\begin{theorem}
\label{thm:nonlinAC}
Let $u^h \in \CM$ and $v \perp u_i^h, \,\, i=1, \ldots, N+1$. Assume that
 $\|v\| < \eps^{1/2 + m}$ for some~$m > 0$. Then, for $\lambda_0$ the constant given in the spectral bound of Theorem~\ref{thm:gapAC},  we obtain
\[
 \langle \CL^h v + \CN^h(v), v \rangle \, \leq \, - \frac12 \lambda_0 \, 
 \Vert v \Vert^2.
\]
\end{theorem}

\begin{proof}
Let $v \perp u_i^h \;\; \forall i = 1,\ldots, N+1$.
By the main spectral result of Theorem~\ref{thm:gapAC}, we have 
\[
\langle \CL^h v, v \rangle \, \leq \, - \lambda_0 \Vert v \Vert^2.
\]
Therefore, for $\gamma_1, \gamma_2 > 0$ with $\gamma_1 + \gamma_2 = 1$, we compute
\begin{equation}
\label{pr:nonlinAC1}
\begin{split}
\langle \CL^h v, v \rangle 
\, &\leq \, - \gamma_1 \lambda_0  \Vert v \Vert^2 + \gamma_2 \eps^2 \int_0^1
v_{xx} \, v \, \mathrm{d}x + \gamma_2 \int_0^1 f'(u^h) v^2 \, \mathrm{d}x  \\
&\leq \, - \gamma_1 \lambda_0 \Vert v \Vert^2
- \eps^2 \gamma_2 \Vert v_x \Vert^2 +
\gamma_2 \Vert f'(u^h) \Vert_{L^\infty} \Vert v \Vert^2.
\end{split}
\end{equation}
By Gagliardo–Nirenberg and Young`s inequality we obtain
\begin{equation}
\label{pr:nonlinAC2}
\begin{split}
\langle \CN^h(v) , v \rangle \, &= \, \int_0^1 3 (u^h)^2 v^3 - v^4 
\, \leq \, 3 \Vert v \Vert_{L^3}^3 \, \leq \, C \Vert v_x \Vert^{1/2} \Vert v \Vert^{5/2}  \\
&\leq \, \eps^2 \gamma_2 \Vert v_x \Vert^2 + C\eps^{-2/3} \gamma_2 ^{-1/3}
\Vert v \Vert^{4/3} \Vert v \Vert^2,
\end{split}
\end{equation}
where we interpolated the $L^3$-norm between $H^1$ and $L^2$. Combining~\eqref{pr:nonlinAC1} and~\eqref{pr:nonlinAC2} yields
\begin{equation*}
\begin{split}
\langle \CL^h v + \CN^h(v) , v \rangle 
\, &\leq \, - \gamma_1 \lambda_0 \Vert v \Vert^2
+ \left[ \gamma_2 \Vert f'(u^h) \Vert_{L^\infty} 
+ C\eps^{-2/3} \gamma_2 ^{-1/3} \Vert v \Vert^{4/3} \right]
\Vert v \Vert^2 \\
&= \, \left[ - \lambda_0 + \gamma_2 \lambda_0 + \gamma_2 \Vert f'(u^h) \Vert_{L^\infty} 
+ C\eps^{-2/3} \gamma_2 ^{-1/3} \Vert v \Vert^{4/3} \right]
\Vert v \Vert^2. 
\end{split}
\end{equation*}
Fixing $\gamma_2 = \eps^m$, we obtain for 
$\|v \| < \eps^{1/2 + m} $ 
\begin{equation*}
\inner{\CL^h v + \CN^h(v)}{v} \, \leq \,
\left[ -  \lambda_0 + 
\eps^m \, \left( \lambda_0 + \| f'(u^h) \|_\infty \right) + C \eps^m \right] \| v \|^2. \qedhere
\end{equation*}

\end{proof}

As a next step, we need to analyze the remaining terms of $d \|v\|^2$. We show that, provided$\| v \|$ is 
sufficiently small, they are of order $\CO(\eta_{\eps})$.
\begin{lemma}
\label{lem:remAC}
Under the same assumptions as in Theorem~\ref{thm:nonlinAC}, we obtain
\[
\inner{\CL(u^h)}{v} \, dt - \frac12 \sum_{i,j} \inner{u^h_{ij}}{v}
\inner{\CQ \sigma_i}{\sigma_j} \, dt + \inner{dv}{dv} \, = \, \CO(\eta_{\eps})\, dt.
\]
\end{lemma}
\begin{proof}
We have $\CL(u^h) = \CO(\exp)$ and, as $\|v\| < \eps^{1/2 + m}$,
\[
\inner{u^h_{ij}}{v} \inner{\CQ \sigma_i}{\sigma_j} \, \leq \,
c \eps^{-3/2} \| v \| \eta_{\eps} \eps^{1/2} \eps^{1/2} \, = \, \CO( \eps^m \eta_{\eps}).
\]
For the  It\^o correction term $\langle dv, dv \rangle$ we see that
\[
\inner{dv}{dv} \, = \, \eta_{\eps} \, dt + \sum_{i,j} \left[ \inner{u^h_i}{u^h_j}
- 2 \inner{\CQ u^h_j}{\sigma_j} \right] \, dt \, = \, \CO(\eta_{\eps}) \, dt.
\]
Here, we utilized that by Proposition~\ref{prop:uh} and Lemma~\ref{lem:estdA}  $\| u^h_{i} \| = \CO(\eps^{-1/2})$, 
$\| u^h_{ij} \| = \CO(\eps^{-3/2})$, and 
$\| \sigma_{i} \| = \CO(\eps^{1/2})$.
\end{proof}
Combining the estimates of Theorem~\ref{thm:nonlinAC} and Lemma~\ref{lem:remAC}, we fully 
estimated the stochastic differential~$d\| v \|^2$. This provides us with the following result, which is essential for proving stability in $L^2$.
\begin{cor}
\label{cor:L2AC}
Let $u^h \in \CM$. If $v \perp u_i^h$ for $ i=1, \ldots, N+1$ and
 $\| v \|< \eps^{1/2 + m}$ for some~
$m > 0$, we obtain
\[
d \|v\|^2 \, \leq \, 
\left[ - \frac12 \lambda_0 \|v\|^2 + \CO(\eta_{\eps})  \right] \, dt 
+ 2 \inner{v}{dW}.
\]
\end{cor}
We can finally show that the $L^2$-norm of $v$ stays small for very
long times under small stochastic perturbations. Since the following stability results can only hold as long as 
${h(t) \in \Omega_{\rho_\eps}}$, we define the first exit time from the open set $\Omega_{\rho_\eps}$ by
\begin{equation}
\label{def:multiAC:firstexit}
\tau_0 \, \coloneqq \, \Big\{ t \, \geq \, 0 \, : \; h(t) \notin \Omega_{\rho_\eps} \Big\}.
\end{equation}
Note that we have seen in Remark~\ref{rem:timescale} that at times of order $\CO(\eps \eta_{\eps}^{-1})$ the interface positions are likely
to move by the magnitude of $\eps$ and thus exit the set of admissible positions~$\Omega_{\rho_\eps}$. This suggests that---if stability holds--- the exit time
$\tau_0$ is with high probability of order~$\eps \eta_{\eps}^{-1}$.

\begin{theorem}[$L^2$-Stability for \eqref{eq:AC}]
\label{thm:ACL2stab}
For $m>0$ define the stopping time
\[
\tau^\ast \, \coloneqq \, \inf \left\{ t \in [ 0,T_\eps \land \tau_0 ] \, : \, \| v(t) \| \, > \, 
\eps^{1/2+m} \right\},
\]
where the deterministic cut-off satisfies $T_\eps = \eps^{-M}$ for fixed large $M>0$ and $\tau_0$ is given by~\eqref{def:multiAC:firstexit}.
Also, assume that for some $\nu \in (0,1)$ 
\[
\| v(0) \| \, \leq \, \nu \eps^{1/2 + m} \quad \text{and} \quad
\eta_{\eps} \, \leq \, \eps^{1+2m}.
\]
Then, the probability $\prob( \tau^\ast < T_\eps \land \tau_0)$ is smaller than any 
power of $\eps$, as $\eps$ tends to zero.
\end{theorem}
\begin{proof}
The statement follows directly by combining the estimate of Corollary~\ref{cor:L2AC} with the general 
stability result of Theorem~\ref{meta:stability}.
\end{proof}
%


\subsection{\texorpdfstring{$L^2$-Stability for \eqref{eq:mAC}}{L2-Stability for (mAC)}}
\label{subsec:L2mAC}

As a next step, we study the $L^2$-stability for the mass conserving Allen--Cahn equation. Since the method of establishing stability works
in a similar fashion to the preceding section, we will only state the main result here. Essentially, the main difference lies in the fact that 
the spectral gap is only of order $\eps$ opposed to a gap of order 
one in the previous case. To compensate this, we need to decrease the region of stability in order to absorb the nonlinear terms 
(compare to Theorem~\ref{thm:nonlinAC}) and thereby, we can only allow 
for a smaller noise strength (cf. Theorem~\ref{meta:stability}).
For more details, we refer to \cite{SchindlerPhD}.
\begin{theorem}[$L^2$-Stability for \eqref{eq:mAC}]
\label{thm:L2stabmac}
For $m>0$ define the stopping time
\[
\tau^\ast \, \coloneqq \, \inf \left\{ t \in [ 0,T_\eps \land \tau_0 ] \, : \, \| v(t) \| \, > \, 
\eps^{3/2+m} \right\},
\]
where $T_\eps = \eps^{-M}$ for fixed large $M>0$ and $\tau_0$ denotes the first exit time from $\CA_{\rho_\eps}$. \\
Also, assume that  
\[
\| v(0) \| \, \leq \, \eps^{3/2 + m} \quad \text{and} \quad
\eta_{\eps} \, \leq \, \eps^{4+2m}.
\]
Then, the probability $\prob( \| v( \tau^\ast) \| > \eps^{3/2 + m})$ is smaller than any 
power of $\eps$, as $\eps$ tends to zero.
\end{theorem}


\subsection{\texorpdfstring{$L^4$-Stability for \eqref{eq:AC}} { L4-Stability for (AC)}}
\label{subsec:L4AC}

For controlling the stochastic ODE of the interface positions, we need to
establish bounds on the nonlinear term
\[
\inner{\CN^h(v)}{u^h_i} \, = \, \int_0^1 (3 u^h  v^2 - v^3 ) \, u^h_i \, \mathrm{d}x.
\]
Since smallness in $L^2$ is not sufficient to control the cubic term, we will prove that the 
$L^4$-norm of $v$ stays small for very long times with high probability. In our analysis, we rely
on the results of the preceding section. There, we established stochastic stability in 
$L^2$ and hence, all constants which appear in the following computations may
depend on $\| v \|_{L^2}$ which---provided the assumptions of Theorem~\ref{thm:ACL2stab} hold true---is smaller than $\eps^{1/2+m}$ for polynomial times in $\eps^{-1}$.

We begin with the classical Allen--Cahn equation \eqref{eq:AC} without mass conservation. By the It\^o formula we have
\[
\frac14 d \|v\|_{L^4}^4 \, = \, \inner{v^3}{dv} + 3 \int_0^1 v^2 (dv)^2 \, \mathrm{d}x.
\]
Again, recall that the flow orthogonal to the slow manifold is given by
\[
dv \, = \, \left[ \CL(u^h) + \CL^h v + \CN^h(v) \right] dt + dW
- \sum_j u^h_j dh_j - \frac12 \sum_{i,j} u_{ij}^h \inner{\CQ \sigma_i}
{\sigma_j} \, dt.
\]
First, let us estimate the It\^o correction term $\int_0^1 v^2 (dv)^2 \, \mathrm{d}x$. 

\begin{lemma}
\label{lem:L4stab:remainder}
Let $h \in \Omega_{\rho_\eps}$. We obtain
\[
\int_0^1 v^2 (dv)^2 \, \mathrm{d}x \, = \, \CO(\eta_{\eps}) \|v\|^2 \, dt.
\]
\end{lemma}

\begin{proof}
Using the relation for $dv$ we see that
\begin{equation*}
\begin{split}
\mathrm{trace}(\CQ) &\int_0^1 v^2 \, dx - 2 \sum_j \int_0^1 v^2 \inner{u_j^h}{\CQ \sigma_j} \,\mathrm{d}x \, dt 
 + \sum_{i,j} \int_0^1 v^2 \inner{u^h_i}{u^h_j} 
\inner{\CQ \sigma_i}{\sigma_j} \, \mathrm{d}x \, dt \\
&\leq \, \eta_{\eps} \|v\|^2 \, dt + c \eps^{-1/2} \eps^{1/2} \eta_{\eps} 
\|v\|^2 \, dt + c\eps^{-1} \eps^{1/2} \eta_{\eps} \eps^{1/2} \|v\|^2 \, dt \, = \, \CO(\eta_{\eps}) \|v\|^2 \, dt, 
\end{split}
\end{equation*}
where we utilized the estimates of Proposition \ref{prop:uh} for the 
derivatives of $u^h$ together with the bound on the diffusion $\sigma$ by
Lemma \ref{lem:estdA}.
\end{proof}

As a next step, we study the critical term $\inner{v^3}{dv}$. 
Expanding $dv$ yields
\begin{equation*}
\begin{split}
\inner{v^3}{dv} \, 
= \, {}&-{} 3 \eps^2 \int_0^1 v^2 v_x^2 \, \mathrm{d}x \, dt - \|v \|_{L^6}^6 \,dt
+ \inner{\CL(u^h)}{v^3} \, dt 
+ \int_0^1 \left( 1-3(u^h)^2 \right) v^4 \,\mathrm{d}x \,dt \\
{}&-{} \int_0^1 3 u^h v^5 \, \mathrm{d}x \, dt
+ \inner{v^3}{dW} - \inner{v^3}{du^h}.
\end{split}
\end{equation*}
We see that the good (negative) terms for our analysis are given by $- \|v\|_{L^6}^6$ and, due to integration by parts, 
\begin{equation*}
\eps^2 \int_0^1 v^3 v_{xx} \, \mathrm{d}x \, = \, - 3\eps^2 \int_0^1 v^2 v_x^2 \, \mathrm{d}x
\, = \, - \frac34 \eps^2 \int_0^1 ((v^2)_x)^2 \, \mathrm{d}x
\, = \, - \frac34 \eps^2 \| (v^2)_x \|^2.
\end{equation*}
Our strategy is to absorb as much as possible of the remaining terms into 
these negative ones, while also using that we can control the $L^2$-norm by the preceding stability result. We begin with analyzing the dominant term.
Since $u^h$ is uniformly bounded, we obtain by interpolating the $L^4$-norm between the good terms
\begin{equation*}
\begin{split}
\int_0^1 \left( 1-3(u^h)^2 \right) v^4 \,\mathrm{d}x \, &\leq \, C \|v\|_{L^4}^4
\, \leq \, C \|v^2\|_{L^\infty} \int_0^1 v^2 \, \mathrm{d}x  \,
\overset{\mathclap{\text{Agmon}}}{\leq} \;\; C \|v\|^2 \|v^2\|_{H^1}^{1/2}
 \|v^2\|^{1/2}   \\
&\overset{\mathclap{\text{Young}}}{\leq} \;\; \frac18 \eps^2 \|v^2 \|_{H^1}^2
+ c \eps^{-2/3} \|v\|^{8/3} \|v\|_{L^4}^{4/3}  \\
&\overset{\mathclap{\text{Hölder}}}{\leq} \;\; \frac18 \eps^2 \|v^2 \|_{H^1}^2
+ c \eps^{-2/3} \|v\|^3 \|v\|_{L^6}  \\
&\overset{\mathclap{\text{Young}}}{\leq} \;\; \frac18 \eps^2 \|v^2 \|_{H^1}^2 
+\frac14 \|v\|_{L^6}^6 + c \eps^{-4/5} \|v\|^{18/5}.
\end{split}
\end{equation*}
Similarly, the $L^5$-term is estimated by
\begin{equation*}
\int_0^1 3 u^h v^5 
\, \leq \, c \|v^2\|_{L^\infty} \|v\|_{L^3}^3 \, \overset{\mathclap{\text{Agmon}}}{\leq} \;\;
c \|v^2\|_{H^1}^{1/2} \|v\|_{L^3}^3 \|v\|_{L^4} 
\overset{\mathclap{\text{Young}}}{\leq} \;\; \frac18 \eps^2 \|v^2 \|_{H^1}^2 
+\frac14 \|v\|_{L^6}^6 + c\eps^{- 4/3} \|v\|^{14/3},
\end{equation*}
where we used Hölder's inequality to interpolate between $L^2$ and $L^6$.
Combining the previous estimates, we derived so far
\begin{equation}
\label{est:v3dv}
\inner{v^3}{dv} \, \leq \, \left[ - \frac12 \eps^2 \|v^2\|_{H^1}^2 \,  
- \frac12 \|v\|_{L^6}^6 + c \eps^{-4/5} \|v\|^{18/5} + c\eps^{- 4/3} \|v\|^{14/3} \right] \, dt 
 {}+{} \inner{v^3}{dW} -\inner{v^3}{du^h}.
\end{equation}
Note that we used $\CL(u^h) = \CO(\exp)$ and thus
$\inner{\CL(u^h)}{v^3}  \leq  \CO(\exp) + \CO(\exp) \|v\|_{L^6}^6$ by Hölder's inequality. 
Finally, we estimate $\inner{v^3}{du^h}$ given by
\[
\inner{v^3}{du^h} \, = \, \sum_j \inner{v^3}{u^h_j} \,dh_j 
+ \frac12 \sum_{i,j} \inner{v^3}{u^h_{ij}}\inner{\CQ \sigma_i}{\sigma_j}
\, dt.
\]
For the second summand we obtain
\begin{equation}
\label{est:secondsummand}
\begin{split}
\inner{v^3}{u^h_{ij}} &\inner{\CQ \sigma_i}{\sigma_j} \, \leq \, 
c \|v^2\|_{L^\infty} \|v\| \eps^{-3/2} \eta_{\eps} \eps^{1/2} \eps^{1/2} 
\quad \overset{\mathclap{\text{Agmon}}}{\leq} \;\; c \|v^2 \|_{H^1}^{1/2} \|v\|_{L^4} \|v\|
\eps^{-1/2} \eta_{\eps} \\
&\overset{\mathclap{\text{Hölder}}}{\leq} \;\; c \|v^2 \|_{H^1}^{1/2} 
\|v\|^{5/4} \|v\|_{L^6}^{3/4}
\eps^{-1/2} \eta_{\eps} 
\, \overset{\mathclap{\text{Young}}}{\leq} \;\; \frac14 \eps^2 \|v^2\|_{H^1}^2 + 
\frac18 \|v\|_{L^6}^6 + c \eps^{-4/5} \eta_{\eps}^{8/5} \|v\|^2.
\end{split}
\end{equation}
In addition to the specified inequalities, we used that $\| u^h_{ij} \| = \CO(\eps^{-3/2})$ by Proposition~\ref{prop:uh}
and~$\| \sigma \| = \CO(\eps^{1/2})$ by Lemma~\ref{lem:estdA}. \\
We conclude by analyzing the term involving the stochastic differential $dh$. Recall that by~\eqref{eq:dh} 
$
dh_j \, = \, b_j(h,v) \, dt + \inner{\sigma_j(h,v)}{dW},
$
where $b$ and $\sigma$ are given by~\eqref{eqAC:b} 
and~\eqref{eqAC:sigma}, respectively.
The diffusion term of $\inner{v^3}{u^h_j} \, dh_j$ can be estimated as 
follows:
\begin{equation*}
\begin{split}
\inner{v^3}{u^h_j} \inner{\sigma_j}{dW} \, = \, 
\inner{ \CO(\|u^h_j\|_{L^\infty} \| \sigma_j \| \|v\|_{L^3}^3)}{dW} 
\, = \, \inner{\CO(\eps^{-1/2} \|v\| \|v\|_{L^4}^2)}{dW}.
\end{split}
\end{equation*}
Estimating the drift of $\inner{v^3}{u^h_j} \, dh_j$ is trickier, as we 
have to bound $b$ as well. By virtue of Lemmata~\ref{lem:estdA} and~\ref{lem:ACestODE},
we can bound the drift $b$ up to a stopping time and obtain as long as$\|v(t)\| < \eps^{1/2 + m}$ for some~${m>0}$
\[
\vert b_j \vert \, = \, \CO(\eta_{\eps}) + \CO(\eps) \, \vert \inner{\CN^h(v)}{u^h_j} \vert \, = \, \CO(\eta_{\eps} + \|v\|^2 + \|v\| \|v\|_{L^4}^2).
\]
Hence, this yields
\begin{equation}
\label{est:prefactortimesb}
\begin{split}
| \inner{v^3}{u^h_j} | |b_j| \, &\leq \, c \eps^{-1} \eta_{\eps} \|v\|_{L^3}^3
+ c \eps^{-1/2} \|v\|^2 \|v\|_{L^6}^3
+ c \eps^{-1/2} \|v\| \|v\|_{L^4}^2 \|v\|_{L^6}^3 \\
&\overset{\mathclap{\text{Hölder}}}{\leq} \;\; c \eps^{-1} \eta_{\eps} \|v\|^ {3/2} \|v\|_{L^6}^{3/2}
+ c \eps^{-1/2} \|v\|^2 \|v\|_{L^6}^3 + c\eps^{-1/2} \|v\|^{3/2}
\|v\|_{L^6}^{9/2} \\
&\overset{\mathclap{\text{Young}}}{\leq} \;\; \frac18 \|v\|_{L^6}^6 
+ c \eps^{-4/3} \eta_{\eps}^{4/3} \|v\|^2 + c\eps^{-1} \|v\|^4 
+ c \eps^{-2} \|v\|^6.
\end{split}
\end{equation}
Finally, we estimated every term of $d \|v\|_{L^4}^4$.
Plugging \eqref{est:secondsummand} and \eqref{est:prefactortimesb} into
\eqref{est:v3dv} furnishes
\[
\inner{v^3}{dv} \, \leq \, \left[ - \frac12 \eps^2 \|v^2\|_{H^1}^2 \,  
- \frac14 \|v\|_{L^6}^6 + K_\eps( \| v \|) \right] \, dt 
 {}+{} \inner{\CO(\eps^{-1/2} \|v\| \|v\|_{L^4}^2)}{dW}, 
\]
where $K_\eps$ is given by
\[
K_\eps(\Vert v \Vert) \, = \, c \eps^{-4/5} \|v\|^{18/5} + c\eps^{- 4/3} \|v\|^{14/3} + c \eps^{-4/3} \eta_{\eps}^{4/3} \|v\|^2 + c\eps^{-1} \|v\|^4 
+ c \eps^{-2} \|v\|^6.
\]

Thus far, the terms in $K_\eps$ depend on the $L^2$-norm of~$v$. 
Under the assumptions of Theorem~\ref{thm:ACL2stab}, i.e., a small noise strength $\eta_{\eps}$ and a suitable initial condition $v(0)$, we can bound $\Vert v \Vert_{L^2}$ by an $\eps$-dependent constant for long time scales. In more detail, we obtain for $\Vert v \Vert \leq \eps^{1/2 + m}$
and $\eta_{\eps} \leq \eps^{1+ 2m}$ that
\[ 
K_\eps(\Vert v \Vert) \, \leq \, \eps^{1 + 2m}.
\]

Noticing now that under the same assumptions the bound in Lemma \ref{lem:L4stab:remainder} provides us with an even smaller term and using the basic estimate 
$ - \| v \|_{L^6}^6 \leq \| v \|^2 - \| v \|_ {L^4}^4$,
we proved the following inequality which is essential for deriving
stochastic stability in $L^4$.

\begin{cor}
\label{thm:ACestL4}
As long as $\Vert v \Vert \leq \eps^{1/2 + m}$ and $\eta_{\eps} \leq \eps^{1 + 2m}$ for some $m >0$, we have 
\begin{equation*}
d\|v\|_{L^4}^4 \, \leq \, \Big[ - \|v\|_{L^4}^4 + c \eps^{1 + 2m} \Big] \, dt 
+ \inner{\CO(\eps^{m}  \|v\|_{L^4}^2}{dW}.
\end{equation*}
\end{cor}

With this inequality at hand, we can apply the main stability theorem \ref{stab:mainest}.
Bear in mind that in the derivation of Theorem~\ref{thm:ACestL4} we presented only one technique and thus,
we cannot guarantee the optimality of the radii. 
\begin{theorem}[$L^4$-Stability for \eqref{eq:AC}]
\label{thm:multiAC:L4stabAC}
For $m>0$ and small $\kappa > 0$, consider the stopping time
\[
\tau^\ast \, = \, \inf \left\{ t \in [0,T_\eps \land \tau_0] \, : \, 
\|v(t)\| \, > \, \eps^{1/2 + m} \quad \text{ or } \quad
\|v(t)\|_{L^4} \, > \, \eps^{1/4 + m/2 - \kappa} \right\},
\]
where $T_\eps = \eps^{-M}$ for any fixed large $M>0$ and $\tau_0$ denotes
the first exit time from $\Omega_{\rho_\eps}$. \\
Also, assume that for some $\nu  \in (0,1)$
\[
\|v(0)\| \, \leq \,\nu \eps^{1/2+m} \quad \text{and} \quad
\|v(0)\|_{L^4} \, \leq \, \nu \eps^{1/4 + m/2 - \kappa}
\]
and that for the squared noise strength 
\[
\eta_{\eps} \, \leq \, \eps^{1+2m}.
\]
Then, the probability $\prob( \tau^\ast < T_\eps \land \tau_0)$ is smaller than any 
power of $\eps$, as $\eps$ tends to zero.
\end{theorem}

\begin{proof}

To fit in the setting of our general stability result of Theorem~\ref{stab:mainest}, we set $x(t) = \Vert v(t) \Vert_{L^4}^4$.
Utilizing the estimate of Corollary~\ref{thm:ACestL4} then yields
\[
dx(t) \, \leq \, \left[ K_\eps( \eta_{\eps}) - a_\eps x(t) \right] \, dt + \langle \CO( c_\eps x(t)^\alpha), \, dW \rangle,
\]
where the constants are given by $a_\eps = 1$, $K_\eps(\eta_{\eps}) = \CO(\eps^{2m+1})$, $c_\eps = \CO(\eps^ m )$, and $\alpha = \tfrac12$. 
Note that due to the substitution and the definition of the stopping time $\tau^\ast$ the radius $R_\eps$ for the variable $x$ is now $\eps^{1+2m -  4 \kappa }$. With that, one easily computes
\[
\frac{K_\eps(\eta_{\eps}) + c_\eps^2 \eta_{\eps}}{a_\eps R_\eps} \, = \, \CO(\eps^{4 \kappa}).
\]
By applying Theorem~\ref{meta:stability}, this shows that $\prob ( \Vert v( \tau^\ast) \Vert_{L^4} > \eps^{1/4 + m/2 - \kappa})$ is smaller than any power of~$\eps$. 
By  the $L^2$-result of Theorem~\ref{thm:ACL2stab} and the basic inequality
\[
\prob \left( \tau_\eps \, < \, T_\eps \land \tau_0 \right) 
\, \leq \, \prob \left( \|v(\tau_\eps)\| \, > \, \eps^{1/2 + m} \right) + \prob \left( \Vert v( \tau^\ast) \Vert_{L^4} > \eps^{1/4 + m/2 - \kappa} \right),
\]
the proof is complete.
\end{proof}


\subsection{\texorpdfstring{$L^4$-Stability for \eqref{eq:mAC}}{L4-Stability for (mAC) }}
\label{subsec:L4mAC}

We conclude our analysis of the stochastic Allen--Cahn equation with 
stating the corresponding $L^4$-stability result for
the mass conserving case \eqref{eq:mAC}. The proof is a straightforward
adaption of the results presented in Section \ref{subsec:L4AC} and is stated in full detail in \cite{SchindlerPhD}.

\begin{theorem}[$L^4$-Stability for \eqref{eq:mAC}]
\label{thm:multiAC:L4mac}
For $m>0$ and small $\kappa > 0$, consider the stopping time
\[
\tau^\ast \, = \, \inf \left\{ t \in [0,T_\eps \land \tau_0] \, : \, 
\|v(t)\| \, > \, \eps^{3/2 + m} \; \text{ or } \;
\|v(t)\|_{L^4} \, > \, \eps^{3/4 + m/2 - \kappa} \right\},
\]
where $T_\eps = \eps^{-M}$ for fixed large $M>0$ and $\tau_0$ denotes 
the first exit time from the set of admissible positions $\CA_{\rho_\eps}$.
Also, assume that for some $\nu \in (0,1)$
\[
\|v(0)\| \, \leq \, \nu \eps^{3/2+m} \quad \text{and} \quad
\|v(0)\|_{L^4} \, \leq \, \nu \eps^{3/4 + m/2 - \kappa}
\]
and that the squared noise strength satisfies
$
\eta_{\eps} \, \leq \, \eps^{4+2m}.
$
\\
Then, the probability $\prob( \tau^\ast < T_\eps \land \tau_0)$ is smaller than any 
power of $\eps$, as $\eps$ tends to zero.
\end{theorem}

\end{document}